\documentclass[11pt,a4paper,leqno]{amsart}

\usepackage[utf8]{inputenc}
\usepackage{amsmath,amssymb,amsfonts,amsthm,yfonts}
\usepackage{latexsym,epsfig,graphicx,subfigure}
\usepackage{mathtools}
\usepackage{esint}
\usepackage{bbm}
\usepackage{calrsfs}
\usepackage{esint}
\usepackage{mathrsfs}

\DeclareFontFamily{U}{mathx}{}
\DeclareFontShape{U}{mathx}{m}{n}{<-> mathx10}{}
\DeclareSymbolFont{mathx}{U}{mathx}{m}{n}
\DeclareMathAccent{\widehat}{0}{mathx}{"70}
\DeclareMathAccent{\widecheck}{0}{mathx}{"71}

\usepackage{graphicx,epsfig,subfigure}

\usepackage{enumitem}
\usepackage{latexsym}
\usepackage{hyperref}
\usepackage{etoolbox}

\makeatletter
\def\thm@space@setup{%
	\thm@preskip=8pt
	\thm@postskip=8pt
}
\makeatother

\AtEndEnvironment{proof}{\vspace{6pt}}

\hypersetup{
	colorlinks=true,
	linkcolor=magenta,
	filecolor=black,
	urlcolor=blue,
	citecolor=cyan,
}

\numberwithin{equation}{section}

\theoremstyle{plain}
\newtheorem*{thm*}{Theorem}
\newtheorem{thm}{Theorem}[section]
\newtheorem{lem}[thm]{Lemma}

\newtheorem*{prop*}{Proposition}
\newtheorem{cor}[thm]{Corollary}

\theoremstyle{definition}
\newtheorem{defn}{Definition}[section]
\newtheorem*{defn*}{Definition}

\theoremstyle{remark}
\newtheorem{rem}{\textbf{Remark}}[section]

\newcommand*{\R}{\mathbb{R}}
\newcommand*{\ve}{\varepsilon}

\DeclareMathAlphabet{\pazocal}{OMS}{zplm}{m}{n}

\makeatletter
\newcommand*{\rom}[1]{\expandafter\@slowromancap\romannumeral #1@}
\makeatother

\begin{document}
	
	\title{$L^2$-boundedness of the $n$-th Calderón commutator on Lipschitz graphs}
	\author{Joan Hernández, Joan Mateu and Laura Prat}\thanks{Joan Hern\'andez and Laura Prat have been supported by PID2024-165107NB-I00 and Joan Mateu by PID2024-155320NB-I00 (Ministerio de Ciencia e Innovación, Spain).}
	\maketitle
	\begin{abstract}
		This paper investigates the asymptotic behavior of the norm, as a bounded operator in $L^2(\R)$, of the $n$-th Calderón commutator $T_{A,n}$ on the graph of a Lipschitz function $A:\R\to\R$. We prove the estimate $\|T_{A,n}\|_{L^2\to L^2} \leq Cn\|A'\|_\infty^n$, thus formalizing a claim by Mateu and Verdera via a symmetrization strategy and the $T1$ theorem. We also show that additional regularity on $A$ yields sublinear growth in $n$. Specifically, for $A$ supported in $[0,1]$, the bound improves to a behavior of the form $\sqrt{n}\|A'\|_\infty^n$ under a Dini condition on $A'$, or if $A'$ belongs to the logarithmic Besov space $B^{1,0}_{1,1}(\R)$. This space contains all compactly supported functions in the Sobolev spaces $H^s(\R)$ for $0<s<1,$ as well as functions of bounded variation. These refined estimates are established through an alternative framework based on Hörmander-type conditions and interpolation, bypassing the standard $T1$ approach. Counterexamples are provided to demonstrate that the Dini and Sobolev fractional regularity conditions are incomparable.
		\bigskip
		
		
		\medskip
		
	\end{abstract}
	\section{Introduction}
	
	The problem of the $L^2$-boundedness of Calderón-type commutators associated with Lipschitz functions is closely related to the boundedness of the Cauchy singular integral on Lipschitz graphs, a central problem in harmonic analysis. Given a Lipschitz graph
	$
	\Gamma=\{x+iA(x):x\in\mathbb R\},
	$
	the Cauchy singular integral is defined by
	$$
	C(f)(z)=\frac1{2\pi i}\,\mathrm{p.v.}\int_{\Gamma}\frac{f(w)}{z-w}\,dw,
	$$
	where $f$ is a function for which the principal value integral exists almost everywhere with respect to the arc-length measure; for example, $f$ can be a compactly supported smooth function on the graph. Its boundedness on \(L^2(\Gamma)\) was established by Coifman, McIntosh and Meyer in 1982 \cite{CMM}, completing earlier work of Calderón \cite{C1,C2}.
	
	Calderón’s starting point in \cite{C1} was the analysis of the first commutator
	$$
	T_{A,1} f(x)=\mathrm{p.v.}\int_{\mathbb{R}}\frac{A(y)-A(x)}{(y-x)^2}f(y)\,dy,
	$$
	where $A$ is a Lipschitz function. Formally, this operator can be expressed as
	\begin{equation}\label{com}
		\frac{-1}{\pi}T_{A,1}
		=
		AH\frac{d}{dx}-H\frac{d}{dx}A
		=
		\bigg[A,H\frac{d}{dx}\bigg],
	\end{equation}
	where $H$ is the Hilbert transform. Calderón proved in \cite{C1} that \(T_{A,1}\) is bounded on \(L^2(\mathbb R)\).
	
	More generally, parametrizing the graph $\Gamma=\{x+iA(x)\}$ leads to the higher-order commutators
	$$
	T_{A,n}f(x)=\mathrm{p.v.}\int_{\mathbb R}
	\left(\frac{A(y)-A(x)}{y-x}\right)^n
	\frac{f(y)}{y-x}\,dy,
	$$
	which arise from the formal expansion of the Cauchy kernel
	\begin{equation}\label{eq}
		\frac{1}{x-y+i(A(x)-A(y))}
		=
		\sum_{n\ge0}(-i)^n
		\frac{(A(x)-A(y))^n}{(x-y)^{n+1}},
	\end{equation}
	valid when \(\|A'\|_\infty<1\). Indeed, one can show that for these higher-order commutators, an analogous formula to \eqref{com} holds, namely
	$$\frac{(-1)^nn!}{\pi}T_{A,n}=\bigg[A,\bigg[A,\dots,\bigg[A,H\frac{d^n}{dx^n}\bigg]\dots\bigg]\bigg].$$
	One way to show this equivalence between operators is to observe that, for $A$ smooth and Lipschitz and $f\in \mathcal{C}^\infty_c(\R)$, the following holds:
	\begin{align*}
		\bigg[A,\bigg[A,\dots,\bigg[A,H\frac{d^n}{dx^n}\bigg]\dots\bigg]\bigg]f(x) &= \mathrm{p.v.} \int \frac{1}{x-y}\sum_{k=0}^{n}(-1)^k \binom{n}{k}A(x)^{n-k}\frac{d^n}{dy^n}(A^kf)(y)\,dy\\
		&=\mathrm{p.v.} \int \frac{1}{x-y}\frac{d^n}{dy^n}\Big[ \big( A(x)-A(y) \big)^nf(y) \Big]\,dy.
	\end{align*}
	Starting from this identity, the desired estimate can be obtained by an induction argument on $n$, based on integration by parts and Taylor's theorem. The result for a general Lipschitz function $A$ is then achieved by applying the smooth result to a mollified approximation $A_\delta$ of $A$, and using the fact that the $L^2$-operator norms of the corresponding truncations of $T_{A_\delta,n}$ are uniformly bounded with respect to the truncation parameter.
	
	From \eqref{eq}, one sees that obtaining quantitative bounds for \(T_{A,n}\) plays a central role in the analysis of the Cauchy integral.
	Indeed, estimates of the form
	$$
	\|T_{A,n}\|_{L^2\to L^2}\le C^n\|A'\|_\infty^n
	$$
	are insufficient to sum the above series unless \(\|A'\|_\infty\) is small, which explains Calderón’s small-slope result in \cite{C2}. To remove this restriction, one needs polynomial control in \(n\):
	$$
	\|T_{A,n}\|_{L^2\to L^2}
	\le C_0(1+n)^k\|A'\|_\infty^n, \quad n=1,2,\ldots
	$$
	for some positive integer $k$. Such polynomial control is sufficient to remove the smallness restriction and recover the full boundedness of the Cauchy singular integral. This was achieved in the celebrated work of Coifman, McIntosh and Meyer \cite{CMM}, who obtained the preceding inequality for \(k=4\). Later, Christ and Journé \cite{christjourne} proved the estimate $$\|T_{A,n}\|_{L^2\to L^2}\le C_\delta(1+n)^{1+\delta}\|A'\|_\infty^n,$$ for every $ \delta>0$. Subsequently, Verdera \cite[Section 5]{V}, building on symmetrization methods related to Menger curvature, indicated that joint work with Mateu yields the sharper estimate
	\begin{equation}\label{u}
		\|T_{A,n}\|_{L^2\to L^2}
		\le C_0(1+n)\|A'\|_\infty^n.
	\end{equation}
	However, a complete proof of \eqref{u} has not appeared in the literature. 
	
	One of the main goals of this paper is to provide a complete proof of \eqref{u}, thereby making rigorous the claim in \cite{V}. We also investigate whether the linear growth in \(n\) is sharp.
	
	For the Cauchy singular integral, it is known that the operator norm grows at most linearly with the Lipschitz constant, and David constructed an example showing that this dependence is sharp \cite{D1} (see Figure \ref{cantor}).
	\begin{figure}
		\label{cantor}
		\centering
		\includegraphics[width=0.85\textwidth]{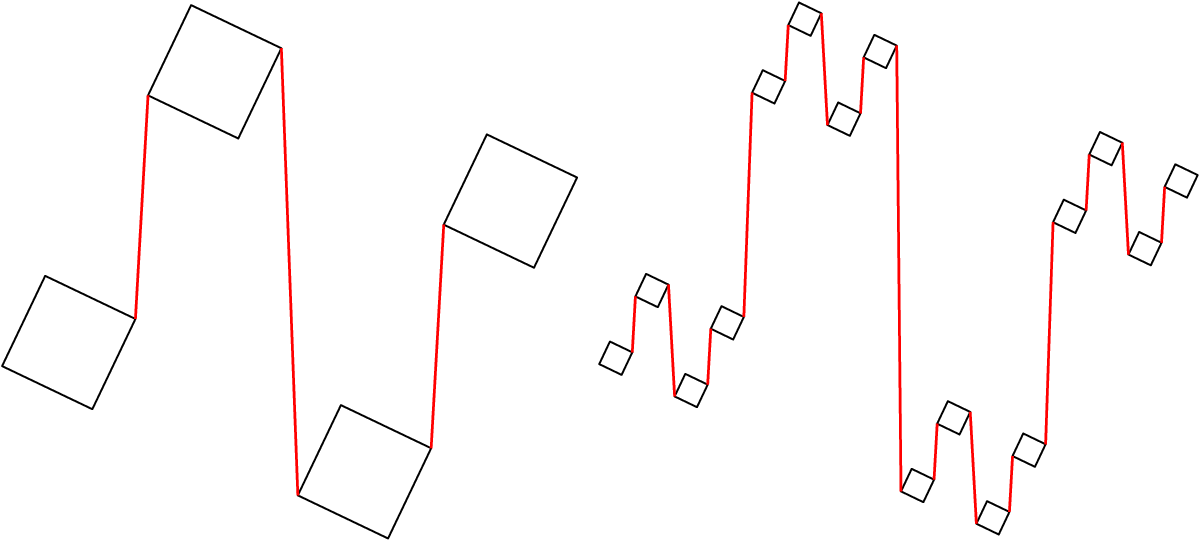}
		\caption{First two iterations involved in the construction of David's saturating Lipschitz function.}
	\end{figure}
	It is therefore natural to ask whether the linear dependence on \(n\) in \eqref{u} is likewise optimal. However, analogous lower bounds for the higher-order commutators \(T_{A,n}\) seem elusive. In all examples we have examined, the lower bounds only reflect the exponential factor \(\|A'\|_\infty^n\), with no polynomial growth in \(n\). This suggests that the bound \eqref{u} may not be optimal in general.
	
	Motivated by this, we study stronger regularity assumptions on \(A\), under which the polynomial growth can be improved. In particular, for functions slightly more regular than Lipschitz, we obtain estimates with growth of order $\sqrt n\|A'\|_\infty^n.$
	
	\section{Notation and statement of the main results}
	
	In this section we introduce the notation used throughout the paper and state the main results.
	
	Let $A:\R\to\R$ be a Lipschitz function. For $\ve>0$, let
	$$
	T_{A,n,\ve}f(x)
	:=
	\int_{|x-y|>\ve}
	\frac{f(y)}{x-y}
	\left(\frac{A(x)-A(y)}{x-y}\right)^n
	\,d y.
	$$
	
	As in the introduction, we denote by $T_{A,n}$ the corresponding principal value operator. The kernel associated to $T_{A,n}$ will be denoted by
	\[
	K_n(x,y):=\frac{1}{x-y}\bigg( \frac{A(x)-A(y)}{x-y} \bigg)^n, \quad x\neq y.
	\]
	If $A$ is Lipschitz, it follows that $K_n$ defines an odd standard kernel and thus $T_{A,n}$ becomes a Calderón-Zygmund operator. Namely, $K_n$ satisfies the pointwise estimates:
	\begin{align*}
		&|K_n(x,y)| \le \frac{\|A'\|_\infty^n}{|x-y|},\\
		&|K_n(x,y)-K_n(x,z)| \le Cn\|A'\|_\infty^n \frac{|y-z|}{|x-y|^{2}}, \qquad \text{ for }\; |x-y|>2|y-z|,
	\end{align*}
	where $C$ is an absolute constant. The constants appearing in the above estimates are known as the (pointwise) Calderón-Zygmund constants of the kernel, and will usually be denoted by $C_{\mathrm{CZ}}$. Note that their dependence on $n$ is linear.
	
	These constants play a crucial role when one seeks to estimate, or optimize, the $L^2(\mathbb{R})$ operator norm of $T_{A,n}$ through a direct application of a $T1$ theorem. However, in the present setting, the constants obtained from the pointwise Calderón-Zygmund estimates are too large to yield sharp bounds for the operator norm. For this reason, we instead rely on integral smoothness estimates of Hörmander type, which provide finer quantitative information. The refined constants arising from these Hörmander-type estimates will also be referred to, when no ambiguity is possible, as Calderón-Zygmund constants of the kernel, and also denoted by $C_{\mathrm{CZ}}$. In order to exploit these improved constants, one must therefore avoid a direct use of the $T1$ theorem and employ alternative methods.
	
	We begin with a general result for Lipschitz functions $A$. 
	
	\begin{thm}\label{main}
		Let $A:\R\to\R$ be a Lipschitz function. Then
		$$
		\|T_{A,n}\|_{L^2\to L^2}
		\le
		Cn\|A'\|_\infty^n.
		$$
		where $C>0$ is an absolute constant.
	\end{thm}
	
	In the results that follow, the Lipschitz function $A$ is assumed to be compactly supported. For simplicity, we assume that $A$ is supported in $[0,1]$. Otherwise, the estimates below should be modified so that an additional factor depending on the length of the support is included. In fact, as argued in \cite[\textsection 3]{V}, one may further assume that $A$ is of class $\mathcal{C}^1$, since the purely Lipschitz case follows by approximation.
	
	We recall that, given a function $f:\mathbb{R}\to\mathbb{R}$, its modulus of continuity is defined by
	$$
	\omega_f(t):=\sup_{|x-y|\le t} |f(x)-f(y)|, \qquad t>0.
	$$
	We say that $\omega_f$ satisfies a Dini condition if
	$$
	C_{\omega_f}:=\int_0^1 \frac{\omega_f(t)}{t}\,dt < \infty.
	$$
	
	\begin{thm}\label{2}
		Let $A$ be a Lipschitz function supported in $[0,1]$ and set
		$\displaystyle
		\omega(t):=\omega_{A'/\|A'\|_\infty}(t).
		$
		If $\omega$ satisfies a Dini condition, then
		$$
		\|T_{A,n}\|_{L^2\to L^2} \leq C\bigl(1+\sqrt{nC_\omega}+\log n\bigr)\|A'\|_\infty^n,
		$$
		where $C>0$ is an absolute constant.
	\end{thm}
	
	In particular, under the assumptions of Theorem \ref{2}, the above bound yields improved control on the growth in $n$. 
	
	For the next result, we recall the definition of the fractional derivative.
	
	\begin{defn}
		Let $f(t,x)\in \mathcal S(\mathbb R^2)$. The fractional derivative of order $\frac12$ with respect to $t$ is defined, via the Fourier transform, by
		\begin{equation*}
			\widehat{\partial_t^{\frac{1}{2}} f}(\tau,\cdot)
			:=
			|\tau|^{\frac{1}{2}} \widehat{f}(\tau,\cdot).
		\end{equation*}
		
	\end{defn}
	
	By Plancherel's identity, this implies
	\begin{equation*}
		\|\partial_t^{\frac{1}{2}} f\|_{L^2(\R^2)}^2
		=
		\int_{\R} |\tau| \|\widehat{f}(\tau,\cdot)\|_{L^2_x}^2 \, d\tau.
	\end{equation*}

	We also prove the following auxiliary result:
	\begin{lem}
		\label{3}
		Let $A$ be a Lipschitz function supported in $[0,1]$ and define the function
		\begin{equation*}
			F_n(t,x):=\frac{1}{\|A'\|_\infty^n}\left( \frac{A(x+t)-A(x)}{t} \right)^n.
		\end{equation*}
		Then,
		\begin{equation*}
			\|T_{A,n}\|_{L^2\to L^2} \leq C\left(\sqrt{n}+\|\partial_t^{\frac{1}{2}} F_n\|_{L^2} \right)\|A'\|_\infty^n,
		\end{equation*}
		where $C>0$ is an absolute constant and $\| \cdot \|_{L^2}$ denotes the $L^2(\R^2)$ norm.
	\end{lem}
	
	For the next theorem, we define the logarithmic Besov space
	\(B^{0,1}_{1,1}(\mathbb R)\) via a Littlewood--Paley decomposition
	(see e.g. \cite{CD1}) by
	\begin{equation}
		\label{norm_Besov}
		\|f\|_{B^{0,1}_{1,1}}
		:=
		\|\Delta_{-1}f\|_{L^1} + \sum_{j\geq 0}(1+j)\|\Delta_j f\|_{L^1}.
	\end{equation}
	Here, for \(j\ge 0\), the dyadic block \(\Delta_j\) is defined by
	$\displaystyle
	\Delta_j f
	=
	\mathcal{F}^{-1}\bigl(\varphi(2^{-j}\xi)\widehat f(\xi)\bigr),
	$
	where \(\varphi\in \mathcal C_c^\infty(\mathbb R\setminus\{0\})\) is supported in
	$
	\left\{\tfrac12\le |\xi|\le 2\right\}.
	$
	The term $\Delta_{-1}f$ accounts for all the low frequency terms. It is not hard to prove that $\|\Delta_{-1}f\|_{L^1}\leq C\|f\|_{L^1}$, where $C$ only depends on the $L^1$ norm of the Schwartz kernel associated with the inverse Fourier transform of the compactly supported multiplier defining $\Delta_{-1}$. Thus, $C$ only depends on $\varphi$ and can be assumed to be 1.
	
	%For the next theorem, we define the logarithmic Besov space
	%\(B^{0,1}_{1,1}(\mathbb R)\) through a Littlewood--Paley decomposition
	%(see e.g. \cite{CD}) by
	%$$
	%\|f\|_{B^{0,1}_{1,1}}
	%:=
	%\sum_{j\ge -1}(1+j_+)\|\Delta_j f\|_{L^1},
	%\qquad j_+=\max\{j,0\},
	%$$ where for $j\geq 0$ the operator $\Delta_j$ is defined through its Fourier symbol $\varphi(2^{-j}\xi)$, with $\varphi\in \mathcal{C}^\infty_c(\R\setminus{\{0\}})$ and support in $\{\frac{1}{2}\leq |\xi|\leq 2\}$, namely given a function $f\in\mathcal{S}$, $\Delta_jf=\mathcal{F}^{-1}(\varphi(2^{-j}\xi)\widehat f)$.
	%	\begin{thm}
		%	\label{5}
		%	Let $A$ be a Lipschitz function supported in $[0,1]$ and such that $A'\in B^{0,1}_{1,1}$. Then,
		%	\begin{equation*}
			%	\|T_{A,n}\|_{L^2\to L^2} \leq C\left( 1+\log n+\sqrt{n}\sqrt{\frac{\|A'\|_{B^{0,1}_{1,1}}}{\|A'\|_\infty} }\right)\|A'\|_\infty^n,
			%\end{equation*}
			%where $C>0$ is an absolute constant.
			%\end{thm}
			\begin{thm}
				\label{5}
				Let $A$ be a Lipschitz function supported in $[0,1]$ such that $A'\in B^{0,1}_{1,1}$. Then
				$$
				\|T_{A,n}\|_{L^2\to L^2}
				\le
				C\sqrt n\|A'\|_\infty^n\left(
				1+\frac{\|A'\|_{B^{0,1}_{1,1}}}{\|A'\|_\infty}
				\right)^{1/2},
				$$
				where $C>0$ is an absolute constant.
			\end{thm}
			
			We conclude with a brief outline of the paper. Section~\ref{secciolipschitz} establishes Theorem~\ref{main} via symmetrization and the $T(1)$ theorem, yielding the best known polynomial growth for the $L^2$ norm of the $n$-th Calderón commutator. Since this approach does not provide sufficiently precise control of the constants needed for the refined estimates, Section \ref{secciointerpolacio} develops an alternative framework based on interpolation and integral estimates of Hörmander type, avoiding the use of pointwise Calderón–Zygmund bounds for the kernel. Although this framework could also be used to recover the general Lipschitz bound, it is introduced primarily to handle the refined results and underlies all subsequent sections. Section~\ref{secciodini} proves Theorem~\ref{2}, obtaining improved sublinear growth under a Dini continuity assumption on $A'$. Section~\ref{secciomig} establishes the auxiliary fractional derivative estimate of Lemma~\ref{3}, and Section~\ref{secciobesov} combines these ingredients to prove Theorem~\ref{5} under logarithmic Besov regularity assumptions. Finally, Section~\ref{examples} shows that the Dini assumptions appearing in Theorem~\ref{2} are incomparable with other fractional Sobolev regularity assumptions that can be deduced from Theorem~\ref{5}; see Corollary~\ref{4} for details.
			
			Throughout the paper, $C>0$ denotes a positive absolute constant whose value may change from one occurrence to another. We write $A\lesssim B$ if there exists a positive constant $C$ such that $A\leq CB.$ Similarly, $A\simeq B$ means that both $A\lesssim B$ and $B\lesssim A$ hold. 
			
			We will also use the standard notation for $L^p$ and $\mathrm{BMO}$ spaces. Moreover, when $f:\mathbb{R}^2\to\mathbb{R}$, the notation $\|f\|_{L^2}$ will always refer to the $L^2(\mathbb{R}^2)$ norm of the function.
			
			Without loss of generality, we shall also normalize $A$ to be $\frac{A}{\|A'\|_\infty}$. That is, throughout the whole text, $A$ will be $1$-Lipschitz. Observe that this can be done due to the identity
			\begin{equation*}
				\|T_{\frac{A}{\|A'\|_\infty},n}\|_{L^2\to L^2} = \frac{1}{\|A'\|_\infty^n}\|T_{A,n}\|_{L^2\to L^2}.
			\end{equation*}
			
			\section{Proof of Theorem \ref{main}}
			\label{secciolipschitz}
			Consider the kernel of the $n$-th Calderón commutator
			$$
			K_n(x,y)
			:=
			\frac{1}{x-y}
			\left(\frac{A(x)-A(y)}{x-y}\right)^n,
			\qquad x\ne y.
			$$
			Then
			$$
			T_{A,n,\ve}f(x)
			=
			\int_{|x-y|>\ve} K_n(x,y)f(y)\,d y.
			$$
			Fix now an interval $I=[a,b]$. Since $T_{A,n}$ is an antisymmetric Calderón--Zygmund operator, the $T(1)$ theorem reduces its $L^2$ boundedness establishing uniform testing estimates on characteristic functions of intervals (see \cite{DJ}). More precisely, it suffices to control quantities of the form
			$$
			\frac{\|T_{A,n,\varepsilon}\chi_I\|_{L^2(I)}^2}{|I|}
			$$
			uniformly in $I$ and $\varepsilon$.
			
			To estimate this quantity, we follow the symmetrization strategy introduced by Verdera in the study of the first Calderón commutator. The idea is to rewrite the testing norm in a form that separates a symmetric contribution from an error term. The latter will be estimated directly, while the former will be analyzed through an algebraic identity revealing an underlying positivity structure.
			
			Observe that for each $0<\ve\ll 1$, we can write
			$$
			\|T_{A,n,\ve}\chi_I\|_{L^2(I)}^2
			=
			\int_I |T_{A,n,\ve}\chi_I(x)|^2\,d x
			=
			\int_I
			\int_{I_\ve(x)}
			\int_{I_\ve(x)}
			K_n(x,y)K_n(x,z)\,d z\,d y\,d x,
			$$
			where
			$$
			I_\ve(x)=\{t\in I:\ |x-t|>\ve\}.
			$$
			
			Equivalently, the above expression can be written as an integral over
			$$
			D_\ve
			:=
			\{(x,y,z)\in I^3:\ |x-y|>\ve,\ |x-z|>\ve\}.
			$$
			
			We now define the symmetric set
			$$
			S_\ve
			:=
			\{(x,y,z)\in I^3:\ |x-y|>\ve,\ |x-z|>\ve,\ |y-z|>\ve\}
			$$
			and its complement inside $D_\ve$,
			$$
			E_\ve
			:=
			D_\ve\setminus S_\ve.
			$$
			
			Then
			\begin{equation}\label{decomposition}
				\|T_{A,n,\ve}\chi_I\|_{L^2(I)}^2
				=
				\iiint_{S_\ve} K_n(x,y)K_n(x,z)\,d z\,d y\,d x
				+
				\iiint_{E_\ve} K_n(x,y)K_n(x,z)\,d z\,d y\,d x.
			\end{equation}

			In the next lemma, we estimate the second term in the right-hand side of \eqref{decomposition}, which should be understood as an error term.
			
			\begin{lem}\label{lem:error}
				$$
				\iiint_{E_\ve}
				|K_n(x,y)K_n(x,z)|\,d z\,d y\,d x
				\le
				C|I|.
				$$
			\end{lem}
			
			\begin{proof}
				Since $A$ is Lipschitz with $\|A'\|_\infty=1$, for all $x\ne y$, $|K_n(x,y)|
				\le
				\frac{1}{|x-y|}$. Hence it follows that
				$$
				\iiint_{E_\ve}
				|K_n(x,y)K_n(x,z)|\,d z\,d y\,d x
				\le
				\iiint_{E_\ve}
				\frac{d x\,d y\,d z}{|x-y||x-z|}.
				$$
				
				Since
				$$
				E_\ve
				=
				\{(x,y,z)\in I^3:\ |x-y|>\ve,\ |x-z|>\ve,\ |y-z|\le \ve\},
				$$
				we can write
				$$
				\iiint_{E_\ve}
				\frac{d x\,d y\,d z}{|x-y||x-z|}
				=
				\int_I
				\int_{|x-z|>\ve}
				\left(
				\int_{\substack{|y-z|\le \ve\\ |x-y|>\ve}}
				\frac{d y}{|x-y|}
				\right)
				\frac{d x\,d z}{|x-z|}.
				$$
				
				We split the last integral into:
				$$
				\iiint_{E_\ve}
				\frac{d x\,d y\,d z}{|x-y||x-z|}
				=
				I_1+I_2,
				$$
				where
				$$
				I_1
				:=
				\int_I
				\int_{|x-z|>2\ve}
				\left(
				\int_{\substack{|y-z|\le \ve\\ |x-y|>\ve}}
				\frac{d y}{|x-y|}
				\right)
				\frac{d x\,d z}{|x-z|},
				$$
				and
				$$
				I_2
				:=
				\int_I
				\int_{\ve<|x-z|\le 2\ve}
				\left(
				\int_{\substack{|y-z|\le \ve\\ |x-y|>\ve}}
				\frac{d y}{|x-y|}
				\right)
				\frac{d x\,d z}{|x-z|}.
				$$
				
				For $I_1$, observe that if $|x-z|>2\ve$ and $|y-z|\le \ve$, then
				$$
				|x-y|
				\ge
				|x-z|-|z-y|
				\ge
				|x-z|-\ve
				\ge
				\frac{|x-z|}{2}.
				$$
				Therefore
				$$
				\int_{\substack{|y-z|\le \ve\\ |x-y|>\ve}}
				\frac{d y}{|x-y|}
				\le
				\frac{2}{|x-z|}
				\cdot |\{y:\ |y-z|\le \ve\}|
				\le
				\frac{4\ve}{|x-z|}.
				$$
				Hence
				$$
				I_1
				\le
				4\ve
				\int_I
				\int_{|x-z|>2\ve}
				\frac{d z\,d x}{|x-z|^2}\le 4|I|.
				$$
				
				For $I_2$, we are in the region
				$
				\ve<|x-z|\le 2\ve.
				$
				Moreover, since $|y-z|\le \ve$, the triangle inequality gives
				$$
				|x-y|\le |x-z|+|z-y|\le 3\ve.
				$$
				Together with the constraint $|x-y|>\ve$, this implies
				$
				\ve<|x-y|\le 3\ve.
				$
				Hence
				$$
				\int_{\substack{|y-z|\le \ve\\ |x-y|>\ve}}
				\frac{d y}{|x-y|}
				\le
				\int_{\ve<|x-y|\le 3\ve}
				\frac{d y}{|x-y|}
				\le
				2\int_{\ve}^{3\ve}\frac{dr}{r}
				\le C.
				$$
				Therefore,
				$$
				I_2
				\le
				C
				\int_I
				\int_{\ve<|x-z|\le 2\ve}
				\frac{d z\,d x}{|x-z|}.
				$$
				For each fixed $x\in I$,
				$$
				\int_{\substack{z\in I\\ \ve<|x-z|\le 2\ve}}
				\frac{d z}{|x-z|}
				\le
				2\int_{\ve}^{2\ve}\frac{dr}{r}
				\le C.
				$$
				Consequently,
				$
				I_2\le C|I|,
				$
				for some absolute constant $C$.
				
				Combining both estimates, we get the desired bound for the error term:
				$$
				\iiint_{E_\ve}
				|K_n(x,y)K_n(x,z)|\,d z\,d y\,d x
				\le
				C|I|.
				$$
			\end{proof}
			We now turn to the symmetrization of the main term in \eqref{decomposition}.
			We begin by writing
			$$
			S_n(x,y,z)
			:=
			K_n(x,y)K_n(x,z)
			-
			K_n(x,y)K_n(y,z)
			+
			K_n(x,z)K_n(y,z).
			$$
			
			so that by symmetry under permutations of the variables, we obtain
			$$
			\iiint_{S_\ve}
			K_n(x,y)K_n(x,z)\,d z\,d y\,d x
			=
			\frac13
			\iiint_{S_\ve}
			S_n(x,y,z)\,d z\,d y\,d x
			=
			2\iiint_{\Sigma_\ve}
			S_n(x,y,z)\,d z\,d y\,d x,
			$$
			where
			$$
			\Sigma_\ve
			:=
			\{(x,y,z)\in I^3:\ a<x<y<z<b,\ y-x>\ve,\ z-y>\ve\}.
			$$
			Combining the first equality above with \eqref{decomposition} and Lemma~\ref{lem:error}, we get
			\begin{equation}\label{eq:unweighted-testing-condition}
				\|T_{A,n,\varepsilon}\chi_I\|_{L^2(I)}^2
				=
				\frac13
				\iiint_{S_\varepsilon}
				S_n(x,y,z)\,dx\,dy\,dz
				+
				O(|I|).
			\end{equation}
			We now rewrite $S_n$ conveniently. Define
			$$
			\alpha:=\frac{A(y)-A(x)}{y-x},\qquad
			\beta:=\frac{A(z)-A(y)}{z-y},\qquad
			\gamma:=\frac{A(z)-A(x)}{z-x}.
			$$
			
			Then $|\alpha|,|\beta|,|\gamma|\le 1$, and if
			$$
			\lambda:=\frac{y-x}{z-x}, \qquad \text{we have} \qquad \gamma=\lambda\alpha+(1-\lambda)\beta.
			$$
			
			\begin{lem}\label{Wn}
				For $x<y<z$, one has
				$$
				S_n(x,y,z)
				=
				\left(\frac{\alpha-\beta}{z-x}\right)^2
				W_n(\alpha,\beta,\gamma),
				$$
				where
				$$
				W_n(\alpha,\beta,\gamma)
				=
				\sum_{k=0}^{n-1}
				\gamma^{n-1-k}
				\sum_{j=0}^{n-k-1}
				\alpha^{n-1-j}\beta^{k+j}=\sum_{\substack{j+k+l=2n-2\\j,k,l\leq n-1}}\alpha^j\beta^k\gamma^l.
				$$
			\end{lem}
			\begin{proof}
				A direct computation shows that
				$$
				S_n(x,y,z)
				=
				\frac{P_n(\alpha,\beta,\lambda)}{(z-x)^2},
				$$
				where
				$$
				P_n(\alpha,\beta,\lambda)
				=
				\frac{\alpha^n\gamma^n}{\lambda}
				-
				\frac{\alpha^n\beta^n}{\lambda(1-\lambda)}
				+
				\frac{\beta^n\gamma^n}{1-\lambda}.
				$$
				If $\alpha=\beta$, then $\gamma=\alpha=\beta$ and the identity is immediate. Thus we may assume $\alpha\neq\beta$. Since $0<\lambda<1$, $\gamma\neq\alpha,\beta$. Using
				$
				\gamma=\lambda\alpha+(1-\lambda)\beta,
				$
				we have
				$$
				\lambda=\frac{\gamma-\beta}{\alpha-\beta},
				\qquad
				1-\lambda=\frac{\alpha-\gamma}{\alpha-\beta}.
				$$
				Hence
				$$
				\begin{aligned}
					P_n(\alpha,\beta,\lambda)
					&=
					\frac{\alpha^n\gamma^n(\alpha-\beta)}{\gamma-\beta}
					-
					\frac{\alpha^n\beta^n(\alpha-\beta)^2}{(\gamma-\beta)(\alpha-\gamma)}
					+
					\frac{\beta^n\gamma^n(\alpha-\beta)}{\alpha-\gamma} \\
					&=
					(\alpha-\beta)
					\left[
					\alpha^n\frac{\gamma^n-\beta^n}{\gamma-\beta}
					-
					\beta^n\frac{\alpha^n-\gamma^n}{\alpha-\gamma}
					\right] \\
					&=
					(\alpha-\beta)
					\left[
					\alpha^n\sum_{k=0}^{n-1}\gamma^{n-1-k}\beta^k
					-
					\beta^n\sum_{k=0}^{n-1}\alpha^{n-1-k}\gamma^k
					\right] \\
					&=
					(\alpha-\beta)
					\sum_{k=0}^{n-1}
					\gamma^{n-1-k}
					\left(
					\alpha^n\beta^k-\beta^n\alpha^k
					\right),
				\end{aligned}
				$$
				where the third equality follows from
				$
				\frac{u^n-v^n}{u-v}
				=
				\sum_{k=0}^{n-1}u^{n-1-k}v^k,
				$
				and the fourth one from reindexing the second sum. Finally,
				$$
				\alpha^n\beta^k-\beta^n\alpha^k
				=
				\alpha^k\beta^k(\alpha^{n-k}-\beta^{n-k})
				=
				(\alpha-\beta)
				\sum_{j=0}^{n-k-1}
				\alpha^{n-1-j}\beta^{k+j}.
				$$
				Therefore
				$$
				P_n(\alpha,\beta,\lambda)
				=
				(\alpha-\beta)^2
				\sum_{k=0}^{n-1}
				\gamma^{n-1-k}
				\sum_{j=0}^{n-k-1}
				\alpha^{n-1-j}\beta^{k+j}.
				$$
				This proves the desired identity.
			\end{proof}

			By the previous lemma,
			$$
			\iiint_{S_\ve} K_n(x,y)K_n(x,z)\,d z\,d y\,d x
			=
			2
			\iiint_{\Sigma_\ve}
			\left(
			\frac{\alpha-\beta}{z-x}
			\right)^2
			W_n(\alpha,\beta,\gamma)
			\,d x\,d y\,d z.
			$$
			
			Since $|\alpha|,|\beta|,|\gamma|\le 1$, each monomial in $W_n$ has absolute value at most $1$, and the number of terms is
			$$
			\sum_{k=0}^{n-1}(n-k)=\frac{n(n+1)}2\le n^2, \qquad \text{so that} \qquad |W_n(\alpha,\beta,\gamma)|\le n^2.
			$$
			Therefore
			$$
			\left|
			\iiint_{\Sigma_\ve}
			\left(
			\frac{\alpha-\beta}{z-x}
			\right)^2
			W_n(\alpha,\beta,\gamma)
			\,d x\,d y\,d z
			\right|
			\le
			n^2\iiint_{\Sigma_\ve}
			\left(
			\frac{\alpha-\beta}{z-x}
			\right)^2
			\,d x\,d y\,d z.
			$$
			
			From the definitions of $\alpha$ and $\beta$, we have
			$$
			\left(
			\frac{\alpha-\beta}{z-x}
			\right)^2
			=
			\left(
			\frac{
				\frac{A(y)-A(x)}{y-x}
				-
				\frac{A(z)-A(y)}{z-y}
			}{z-x}
			\right)^2.
			$$
			Since the integrand is nonnegative, we may estimate the previous quantity by extending the domain of integration to $\R^3$, and then apply the following identity due to Melnikov and Verdera (see \cite{MV} or \cite{V}).
			
			\begin{lem}\label{lem:triple}
				If $B:\R\to\R$ is Lipschitz and compactly supported, then
				$$
				\iiint_{\R^3}
				\left(
				\frac{
					\frac{B(y)-B(x)}{y-x}
					-
					\frac{B(z)-B(y)}{z-y}
				}{z-x}
				\right)^2
				d x\,d y\,d z
				=
				2\pi^2\int_{\R}|B'(t)|^2\,d t.
				$$
			\end{lem}
			
			Set
			$
			\widetilde A(t)
			:=
			\chi_I(t)\bigl(A(t)-A(a)-\langle A'\rangle_I(t-a)\bigr),
			$
			where
			$\displaystyle
			\langle A'\rangle_I:=\frac{A(b)-A(a)}{b-a}.
			$
			
			By construction, $\widetilde A$ is Lipschitz and compactly supported in $I=[a,b]$. Applying Lemma~\ref{lem:triple} with $B=\widetilde A$, we obtain
			$$
			\iiint_{\R^3}
			\left(
			\frac{
				\frac{\widetilde A(y)-\widetilde A(x)}{y-x}
				-
				\frac{\widetilde A(z)-\widetilde A(y)}{z-y}
			}{z-x}
			\right)^2
			d x\,d y\,d z
			=
			2\pi^2\int_I |A'(t)-\langle A'\rangle_I|^2\,d t.
			$$
			
			Consequently,
			$$
			\left|
			\iiint_{S_\ve} K_n(x,y)K_n(x,z)\,d z\,d y\,d x
			\right|
			\le
			C\,n^2
			\int_I |A'(t)-\langle A'\rangle_I|^2\,d t.
			$$
			Using \eqref{eq:unweighted-testing-condition}, together with the above estimate for the symmetric term and the bound of Lemma~\ref{lem:error} for the error term, we conclude that
			
			$$
			\frac{\|T_{A,n}\chi_I\|_{L^2(I)}^2}{|I|}
			\le
			C
			\left[
			1+
			n^2
			\left(
			\frac1{|I|}
			\int_I |A'(t)-\langle A'\rangle_I|^2\,d t
			\right)
			\right].
			$$
			
			Taking the supremum over all intervals $I$ and applying the $T1$ theorem for antisymmetric Calderón--Zygmund kernels, we obtain
			\begin{align}
				\label{eq_aux3.3}
				\|T_{A,n}\|_{L^2\to L^2}
				&\leq
				C_{\mathrm{CZ}}
				+
				C
				\sqrt{
					1+
					n^2\,
					\sup_I
					\left(
					\frac1{|I|}
					\int_I |A'(t)-\langle A'\rangle_I|^2\,d t
					\right)
				},
			\end{align}
			where $C_{\mathrm{CZ}}$ denotes the Calderón--Zygmund constant of $K_n$. A standard computation on the kernel derivative shows that $C_{\mathrm{CZ}}\lesssim n$. Using that $A$ is 1-Lipschitz we finally get
			$$
			\|T_{A,n}\|_{L^2\to L^2}
			\le
			Cn.
			$$
			and the proof of Theorem~\ref{main} is complete.
			
			We note that the estimate $C_{\mathrm{CZ}}\lesssim n$ is sufficient for Theorem \ref{main}, but is far from optimal. Indeed, consider the affine function $A(x)=x+b$, with $b\in\mathbb{R}$ fixed. In this case, the operator $T_{A,n}$ becomes the Hilbert transform $H$. Since $H$ is an isometry in $L^2(\R)$ we simply have $\|T_{A,n}\|_{L^2\to L^2}=1=\|A'\|_\infty^n$, with no dependence on $n$.
			
			On the other hand, if one followed the indirect approach based on the estimate obtained above, then, since the integral in \eqref{eq_aux3.3} vanishes in the affine case, the only dependence on $n$ would come from the pointwise Calderón-Zygmund constants. This would yield a linear dependence on $n$ for the estimate of $\|T_{A,n}\|_{L^2\to L^2}$.
			
			This already shows that, if one aims to optimize this constant with respect to $n$, a direct application of a $T1$ theorem is not the most efficient approach. In the next section, we therefore obtain sharper bounds for the relevant Calderón-Zygmund constants by means of integral-type estimates. These refined estimates will be essential for the proof of the remaining results.

			\section{The Hörmander condition and interpolation}
			\label{secciointerpolacio}
			This section is devoted to establishing several lemmas, which will be used in the proofs of Theorem \ref{2}, Lemma \ref{3}, Corollary \ref{4} and Theorem \ref{5}. The purpose is to bypass the use of the $T(1)$ theorem and ensure that the Calderón--Zygmund constants in the $L^2$ estimates grow at most like $\log{n}$.
			
			We begin by proving that, for odd $n$, $\displaystyle W_n(\alpha,\beta,\gamma)\ge 0.$
			This stronger positivity property was not needed in the argument above, since the proof only uses the combinatorial estimate
			$$
			|W_n(\alpha,\beta,\gamma)|\le n^2.
			$$ However, it will be essential in the sequel to find better estimates for the Calderón--Zygmund constants of $K_n$. Moreover, due to this positivity issue, 
			the quantity
			$$
			\left(\frac{\alpha-\beta}{z-x}\right)^2 W_n(\alpha,\beta,\gamma)
			$$
			may be viewed as a weighted version of the Menger curvature, which plays a central role in the work of Melnikov and Verdera on the \(L^2\)-boundedness of the Cauchy integral on Lipschitz curves (see \cite{MV,V}). For this reason, we believe that this positivity phenomenon deserves to be highlighted. Finally, although it only appears in the odd case, this does not constitute a genuine obstruction in view of the reduction formula in \cite[p.\,50]{D3},
			$$
			T_{A,n+1}(1)=T_{A,n}(A'),
			$$
			which allows one to pass from estimates for odd commutators to estimates for even commutators.
			
			\begin{lem}
				Assume that $n$ is odd. Then
				$
				W_n(\alpha,\beta,\gamma)\ge 0,
				$
				%	where $\gamma=\lambda\alpha+(1-\lambda)\beta$, with $0<\lambda<1$.
			\end{lem}
			\begin{proof}
				We distinguish two cases. If $\alpha\beta\ge 0$, then $\alpha,\beta$
				and $\gamma$ have the same sign, possibly allowing one of them to be zero.
				Each monomial in $W_n$ has total degree
				$$
				(n-1-k)+(n-1-j)+(k+j)=2n-2,
				$$
				which is even. Hence every term in the definition of $W_n$ is nonnegative,
				and therefore $W_n(\alpha,\beta,\gamma)\ge 0$.
				
				Assume now that $\alpha\beta<0$. Recall that
				$$
				(\alpha-\beta)^2W_n(\alpha,\beta,\gamma)
				=
				(\alpha-\beta)
				\left[
				\alpha^n\frac{\beta^n-\gamma^n}{\beta-\gamma}
				-
				\beta^n\frac{\alpha^n-\gamma^n}{\alpha-\gamma}
				\right].
				$$
				For fixed $\gamma$, define
				$$
				\phi(t):=\frac{t^n-\gamma^n}{t-\gamma}, \quad t\neq \gamma,
				$$
				and also set $\phi(\gamma):=n\gamma^{n-1}$. Since $n$ is odd, the function
				$t\mapsto t^n$ is increasing on $\mathbb R$, and hence $\phi(t)\ge 0$ for
				all $t\in\mathbb R$.
				
				If $\alpha>0>\beta$, then $\alpha-\beta>0$, $\alpha^n>0$ and $\beta^n<0$.
				Thus
				$
				\alpha^n\phi(\beta)-\beta^n\phi(\alpha)\ge 0,
				$
				and consequently
				$
				(\alpha-\beta)^2W_n(\alpha,\beta,\gamma)\ge0.
				$
				
				If $\beta>0>\alpha$, then $\alpha-\beta<0$, $\alpha^n<0$ and $\beta^n>0$.
				Thus
				$
				\alpha^n\phi(\beta)-\beta^n\phi(\alpha)\le 0,
				$
				and again
				$
				(\alpha-\beta)^2W_n(\alpha,\beta,\gamma)\ge0.
				$
				
				Since $\alpha\ne\beta$ in the case $\alpha\beta<0$, we conclude that
				$W_n(\alpha,\beta,\gamma)\ge0$.
			\end{proof}
			Next we prove a Hörmander condition for the kernels $K_n$, with constants growing at most logarithmically in $n$. Recall that we are assuming that $A$ is $1$-Lipschitz, i.e. $\|A'\|_\infty = 1$.
			\begin{lem}\label{hormander}
				There exists an absolute constant $C>0$ such that, for all $y,z\in\mathbb R$,
				$$
				\int_{|x-y|>2|y-z|}
				|K_n(x,y)-K_n(x,z)|\,dx
				\le
				C(1+\log n).
				$$
			\end{lem}
			\begin{proof}
				Define	
				\begin{equation*}
					Q_y(x):=\frac{A(x)-A(y)}{x-y}, \quad \text{so that} \quad  K_n(x,y)=\frac{Q_y(x)^n}{x-y}.
				\end{equation*}
				Since \(A\) is Lipschitz with constant 1, $\displaystyle	|Q_y(x)|\le 1$.
				Let \(h:=|y-z|\). If \(h=0\), the integral is identically 0, so assume \(h>0\). We estimate
				\begin{equation*}
					I(y,z):= \int_{|x-y|>2h} |K_n(x,y)-K_n(x,z)|d x .
				\end{equation*}
				We decompose
				\begin{equation*}
					K_n(x,y)-K_n(x,z) = Q_y(x)^n \left(\frac1{x-y}-\frac1{x-z}\right) + \frac{Q_y(x)^n-Q_z(x)^n}{x-z},
				\end{equation*}
				so that
				\begin{align*}
					I(y,z) &\le \int_{|x-y|>2h} |Q_y(x)|^n \left|\frac1{x-y}-\frac1{x-z}\right|d x + \int_{|x-y|>2h}
					\frac{|Q_y(x)^n-Q_z(x)^n|}{|x-z|}d x \\
					&=:J_1+J_2.
				\end{align*}	
				For \(J_1\), using \(|Q_y(x)|\le 1\), we get
				\begin{equation*}
					J_1 \le \int_{|x-y|>2h} \frac{h}{|x-y||x-z|}d x \leq C.
				\end{equation*}
				We now estimate \(J_2\). For \(|a|,|b|\le 1\), one has $|a^n-b^n|
				\le n|a-b|$ as well as $|a^n-b^n|\le 2$. Therefore,
				\begin{equation*}
					|a^n-b^n| \le \min\left\{ 2,n|a-b| \right\}.
				\end{equation*}
				Applying this with \(a=Q_y(x)\) and \(b=Q_z(x)\), it remains to bound $|Q_y(x)-Q_z(x)|$. We write
				\begin{equation*}
					Q_y(x)-Q_z(x) = (A(x)-A(z)) \left(\frac1{x-y}-\frac1{x-z}\right) + \frac{A(z)-A(y)}{x-y}.
				\end{equation*}
				Using the Lipschitz bound for $A$, we obtain
				\begin{equation*}
					|Q_y(x)-Q_z(x)| \le  \frac{2h}{|x-y|}.
				\end{equation*}
				On the region $|x-y|>2h$, we have $|x-z|\le |x-y|+h < \frac32 |x-y|$, and hence $\frac1{|x-y|} \le \frac32\frac1{|x-z|}$. Consequently,
				\begin{equation*}
					|Q_y(x)-Q_z(x)| \le \frac{3h}{|x-z|},
				\end{equation*}
				and then
				\begin{equation*}
					|Q_y(x)^n-Q_z(x)^n| \le \min\left\{2,\frac{3nh}{|x-z|}\right\}.
				\end{equation*}
				It follows that
				\begin{align*}
					J_2 &\le \int_{|x-y|>2h} \frac1{|x-z|} \min\left\{ 2,\frac{3nh}{|x-z|}\right\}d x \leq 2\int_h^\infty \frac1r \min\left\{2,\frac{3nh}{r}\right\}d r \\
					&\leq C \left( 1+\log n \right). 
				\end{align*}	
				Combining the estimates for \(J_1\) and \(J_2\), we obtain
				$$
				\int_{|x-y|>2|y-z|} |K_n(x,y)-K_n(x,z)|d x \le  C (1+\log{n}).
				$$
				This proves the lemma.
			\end{proof}
			
			\begin{lem}\label{propo}
				Assume that $n$ is odd and suppose that there exists a constant $\Lambda_n$ such that, for every interval $I\subset\R$,
				$$
				\|T_{A,n,\varepsilon}\chi_I\|_{L^2(I)}
				\le
				\Lambda_n |I|^{1/2}
				$$
				uniformly in $\varepsilon>0$. Then
				$$
				\|T_{A,n,\varepsilon}\|_{L^2 \to L^2}\le C(\Lambda_n + 1 + \log n)$$ uniformly in $\varepsilon>0$.
			\end{lem}
			
			\begin{proof}
				Let $b$ be a real function in $L^\infty(I)$. We first show that
				\begin{equation}\label{L2b}
					\|T_{A,n,\varepsilon}(b\chi_I)\|_{L^2(I)}\le
					C\Lambda_n\|b\|_\infty |I|^{1/2}.
				\end{equation} 
				This estimate will imply that $T_{A,n,\varepsilon}$ is bounded from $L^\infty$ to $\mathrm{BMO}$ with constant proportional to $\Lambda_n+1+\log{n}$, and from $H^1$ to $L^1$. Then, by interpolation we will obtain the result.
				
				Denote
				$$
				T_\varepsilon b:=T_{A,n,\varepsilon}(b\chi_I),
				\qquad
				T_\varepsilon 1:=T_{A,n,\varepsilon}(\chi_I).
				$$
				Observe that $K_n(y,x)=-K_n(x,y)$. We first establish the following weighted symmetrization identity:
				\begin{equation}\label{eq:weighted-symmetrization}
					\begin{aligned}
						2&\int_I |T_\varepsilon b(x)|^2\,dx
						+
						4\int_I T_\varepsilon b(x)\,T_\varepsilon 1(x)\,b(x)\,dx  \\
						&=
						\frac{2}{3}
						\iiint_{S_\varepsilon}
						S_n(x,y,z)
						\bigl(
						b(x)b(y)+b(x)b(z)+b(y)b(z)
						\bigr)
						\,dx\,dy\,dz 
						+
						O\bigl(\|b\|_\infty^2 |I|\bigr),
					\end{aligned}
				\end{equation}
				where
				$$
				S_n(x,y,z)
				=
				K_n(x,y)K_n(x,z)
				-
				K_n(x,y)K_n(y,z)
				+
				K_n(x,z)K_n(y,z).
				$$
				
				Indeed, writing $D_\varepsilon
				=
				\{(x,y,z)\in I^3:\ |x-y|>\varepsilon,\ |x-z|>\varepsilon\}$, we have by definition
				$$
				\int_I |T_\varepsilon b(x)|^2\,dx
				=
				\iiint_{D_\varepsilon}
				K_n(x,y)K_n(x,z)b(y)b(z)\,dx\,dy\,dz.
				$$
				We decompose $D_\varepsilon=S_\varepsilon\cup E_\varepsilon$, with
				$$
				S_\varepsilon
				=
				\{(x,y,z)\in I^3:\ |x-y|>\varepsilon,\ |x-z|>\varepsilon,\ |y-z|>\varepsilon\}.
				$$
				
				Using $|K_n(x,y)|\le \frac{1}{|x-y|}$ we obtain
				$$
				\left|
				\iiint_{E_\varepsilon}
				K_n(x,y)K_n(x,z)b(y)b(z)\,dx\,dy\,dz
				\right|
				\le
				\|b\|_\infty^2
				\iiint_{E_\varepsilon}
				\frac{dx\,dy\,dz}{|x-y||x-z|}.
				$$
				As in the unweighted case, the last integral is bounded by $C|I|$, and therefore
				$$
				\left|
				\iiint_{E_\varepsilon}
				K_n(x,y)K_n(x,z)b(y)b(z)\,dx\,dy\,dz
				\right|
				\le
				C\|b\|_\infty^2 |I|.
				$$
				
				Arguing in the same way, the contribution of the nonsymmetric region in the mixed term
				$$
				\int_I T_\varepsilon b(x)\,T_\varepsilon 1(x)\,b(x)\,dx
				$$
				is also bounded by $C\|b\|_\infty^2|I|$.
				
				We now restrict the integrals to $S_\varepsilon$ and exploit its symmetry under permutations of the variables. Define
				$$
				\mathcal{S}_b(x,y,z)
				=
				K_n(x,y)K_n(x,z)b(y)b(z)
				-
				K_n(x,y)K_n(y,z)b(x)b(z)
				+
				K_n(x,z)K_n(y,z)b(x)b(y).
				$$
				Then
				\begin{equation}\label{eq:square}
					\int_I |T_\varepsilon b(x)|^2\,dx
					=
					\frac13
					\iiint_{S_\varepsilon}
					\mathcal{S}_b(x,y,z)\,dx\,dy\,dz
					+
					O(\|b\|_\infty^2 |I|).
				\end{equation}
				
				Similarly, a symmetrization of the mixed term yields
				\begin{equation}\label{eq:mixed}
					\int_I T_\varepsilon b(x)\,T_\varepsilon 1(x)\,b(x)\,dx
					=
					\frac16
					\iiint_{S_\varepsilon}
					\mathcal{M}_b(x,y,z)\,dx\,dy\,dz
					+
					O(\|b\|_\infty^2 |I|),
				\end{equation}
				where
				$$
				\begin{aligned}
					\mathcal{M}_b(x,y,z)
					={}&
					K_n(x,y)K_n(x,z)b(x)b(y)
					+
					K_n(x,y)K_n(x,z)b(x)b(z)\\
					&-
					K_n(x,y)K_n(y,z)b(x)b(y)
					-
					K_n(x,y)K_n(y,z)b(y)b(z)\\
					&+
					K_n(x,z)K_n(y,z)b(x)b(z)
					+
					K_n(x,z)K_n(y,z)b(y)b(z).
				\end{aligned}
				$$
				
				We now combine \eqref{eq:square} and \eqref{eq:mixed}. Set
				$$
				A:=K_n(x,y)K_n(x,z),\quad
				B:=-K_n(x,y)K_n(y,z),\quad
				C:=K_n(x,z)K_n(y,z),
				$$
				and
				$$
				p:=b(x)b(y),\quad q:=b(x)b(z),\quad r:=b(y)b(z).
				$$
				Then
				$
				S_n(x,y,z)=A+B+C,
				$
				and the symmetrized expressions become
				$$
				\mathcal{S}_b = Ar + Bq + Cp,
				\qquad
				\mathcal{M}_b = A(p+q) + B(p+r) + C(q+r).
				$$
				
				Multiplying \eqref{eq:square} by $2$ and \eqref{eq:mixed} by $4$, %we obtain
				%$$
				%2(\cdots) + 4(\cdots)
				%=
				%\frac23
				%\Big[
				%Ar+Bq+Cp
				%+
				%A(p+q)+B(p+r)+C(q+r)
				%\Big].
				%$$
				and grouping terms appropriately yields
				$$
				\frac23
				\Big[
				A(p+q+r)+B(p+q+r)+C(p+q+r)
				\Big]
				=
				\frac23 (A+B+C)(p+q+r),
				$$
				which is exactly
				$$
				\frac23
				S_n(x,y,z)
				\bigl(
				b(x)b(y)+b(x)b(z)+b(y)b(z)
				\bigr),
				$$
				and proves \eqref{eq:weighted-symmetrization}.

				We now use the positivity available for odd $n$. Recall that
				$$
				S_n(x,y,z)
				= \left( \frac{\alpha-\beta}{z-x} \right)^2 W_n(\alpha,\beta,\gamma).
				$$
				Since $n$ is odd, $W_n(\alpha,\beta,\gamma)\ge 0$, and hence $S_n(x,y,z)\ge0$. Therefore,
				$$
				\iiint_{S_\varepsilon}
				S_n(x,y,z)
				\bigl(
				b(x)b(y)+b(x)b(z)+b(y)b(z)
				\bigr)
				\,dx\,dy\,dz
				\le
				3\|b\|_\infty^2
				\iiint_{S_\varepsilon}
				S_n(x,y,z)\,dx\,dy\,dz.
				$$
				
				Enlarging $\Lambda_n$ if necessary, we may assume that $\Lambda_n \ge 1$. Using the unweighted symmetrization identity \eqref{eq:unweighted-testing-condition} and the hypothesis on $\Lambda_n$, we obtain
				$$
				\iiint_{S_\varepsilon} S_n(x,y,z)\,dx\,dy\,dz
				\lesssim
				\|T_{A,n,\varepsilon}\chi_I\|_{L^2(I)}^2 + |I|
				\lesssim
				\Lambda_n^2 |I|.
				$$
				Plugging this estimate into \eqref{eq:weighted-symmetrization} and using Cauchy--Schwarz for the mixed term, we get
				$$
				\|T_\varepsilon b\|_{L^2(I)}^2
				\le
				C\Lambda_n^2 \|b\|_\infty^2 |I|
				+
				C\Lambda_n \|b\|_\infty |I|^{1/2}
				\|T_\varepsilon b\|_{L^2(I)}.
				$$
				Solving this quadratic inequality in $\|T_\varepsilon b\|_{L^2(I)}$ gives \eqref{L2b}. 
				
				By standard Calderón--Zygmund arguments, the above local estimate implies that 
				$T_{A,n,\varepsilon}$ maps $L^\infty(\mathbb{R})$ into $\mathrm{BMO}(\mathbb{R})$, with bound
				\begin{equation}\label{Lbmo}
					\|T_{A,n,\varepsilon}\|_{L^\infty \to \mathrm{BMO}}
					\le
					C(\Lambda_n + 1 + \log n).
				\end{equation}
				
				We briefly justify \eqref{Lbmo}. It is enough to consider real-valued functions. Let $b\in L^\infty(\mathbb R)$ and let $I$ be an interval. We write
				$$
				b=b\chi_{4I}+b\chi_{(4I)^c}.
				$$
				Let $x_I$ denote the center of $I$, and set
				$
				c_I:=T_{A,n,\varepsilon}(b\chi_{(4I)^c})(x_I).
				$
				Then
				$$
				\frac1{|I|}\int_I |T_{A,n,\varepsilon}b(x)-c_I|\,dx
				\le A_I+B_I,
				$$
				where
				$$
				A_I:=
				\frac1{|I|}\int_I |T_{A,n,\varepsilon}(b\chi_{4I})(x)|\,dx
				$$
				and
				$$
				B_I:=
				\frac1{|I|}\int_I
				\left|
				T_{A,n,\varepsilon}(b\chi_{(4I)^c})(x)
				-
				T_{A,n,\varepsilon}(b\chi_{(4I)^c})(x_I)
				\right|\,dx.
				$$
				
				For $A_I$, by Cauchy--Schwarz and the local estimate \eqref{L2b} applied to the interval $4I$, we have
				$$
				A_I
				\le
				\frac{1}{|I|^{1/2}}
				\|T_{A,n,\varepsilon}(b\chi_{4I})\|_{L^2(I)}
				\le
				C\Lambda_n \|b\|_\infty.
				$$
				
				For $B_I$, we use Lemma~\ref{hormander} together with the antisymmetry of the kernel. Since $x\in I$ and $y\in (4I)^c$, the points $x$ and $x_I$ are sufficiently separated from $y$, and hence
				$$
				\int_{(4I)^c}|K_n(x,y)-K_n(x_I,y)|\,dy
				\le
				C(1+\log n).
				$$
				%For the truncated kernels, the difference of truncation regions gives only an additional annular term, which is bounded by $CM^n$ by the size estimate of the kernel. 
				Therefore, $\displaystyle
				B_I
				\le
				C(1+\log n)\|b\|_\infty.
				$
				
				Combining the estimates for $A_I$ and $B_I$ and taking the supremum over all intervals $I$, we conclude that
				$$
				\|T_{A,n,\varepsilon}b\|_{\mathrm{BMO}}
				\le
				C(\Lambda_n+1+\log n)\|b\|_\infty.
				$$
				This proves \eqref{Lbmo}.
				
				Since the kernel is antisymmetric, the adjoint operator satisfies
				$T_{A,n,\varepsilon}^*=-T_{A,n,\varepsilon}$. Hence, the same bound holds for the adjoint with the same constant, and by duality we obtain
				$$
				\|T_{A,n,\varepsilon}\|_{H^1 \to L^1}
				\le
				C(\Lambda_n + 1 + \log n).
				$$
				Finally, by the Fefferman-Stein interpolation theorem between
				$
				H^1\to L^1
				$
				and
				$
				L^\infty\to \mathrm{BMO},
				$ \cite{FS},
				we get
				$\displaystyle
				\|T_{A,n,\varepsilon}\|_{L^2 \to L^2}
				\le
				C(\Lambda_n + 1 + \log n),
				$ and this proves the lemma.
			\end{proof}
			
			\section{Proof of Theorem \ref{2}: the Dini case}
			\label{secciodini}
			
			In this section we prove Theorem~\ref{2} when $\|A'\|_\infty=1$, that can be assumed without loss of generality. By the remark at the beginning of the previous section, it is enough to prove the estimate for odd values of \(n\), since the even case follows from the identity
			$\displaystyle
			T_{A,n+1,\varepsilon}(1)=T_{A,n,\varepsilon}(A')
			$
			together with Lemma~\ref{propo}. Accordingly, throughout this section we assume that \(n\) is odd, so that the positivity property
			$
			W_n(\alpha,\beta,\gamma)\ge 0
			$
			established in the previous section is available. 
			
			Recall that Lemma \ref{propo} yields
			$$
			\|T_{A,n}\|_{L^2\to L^2}
			\lesssim
			(\Lambda_n+1+\log n),
			$$
			provided that
			$$
			\|T_{A,n,\varepsilon}\chi_I\|_{L^2(I)}
			\le
			\Lambda_n |I|^{1/2}
			$$
			uniformly over all intervals $I$ and all $\varepsilon>0$.
			
			In the purely Lipschitz setting, one obtains $\Lambda_n\lesssim n$. Our goal here is to improve this growth under the additional Dini regularity assumption on $A'$.  
			
			The next lemma uses the compact support of \(A\) to reduce the testing condition, up to a harmless error, to intervals contained in a fixed neighbourhood of the support.
			\begin{lem}
				\label{lem1.1}
				Let $A$ be a $1$-Lipschitz function supported in $[0,1]$. Then, there exists an absolute constant $C>0$ such that
				\begin{equation*}
					\sup_{I\subset \R} \frac{\|T_{A,n,\ve}\chi_I\|_{L^2(I)}}{|I|^{1/2}} \leq \sup_{J\subset [-1,2]} \frac{\|T_{A,n,\ve}\chi_J\|_{L^2(J)}}{|J|^{1/2}}  + C.
				\end{equation*}
			\end{lem}
			
			\begin{proof}
				
				Let $S:=[0,1]$ and $S_1:=[-1,2]$. Fix any interval $I\subset \R$ and write $J:=I\cap S_1$. It is clear that if $I\cap S = \varnothing$, then for each $x\in I$,
				\begin{equation*}
					T_{A,n,\ve} \chi_I (x) = \int_{I\cap\{|x-y|>\ve\}} \left( \frac{A(x)-A(y)}{x-y} \right)^n \frac{dy}{x-y} = 0,
				\end{equation*}
				and then $\|T_{A,n,\ve}\chi_I\|_{L^2(I)}=0$. So we assume $I\cap S \neq \varnothing$, so that $J\neq \varnothing$ with $|J|>0$. By the triangle inequality,
				\begin{equation*}
					\|T_{A,n,\ve}\chi_I\|_{L^2(I)} \leq \|T_{A,n,\ve}\chi_I\|_{L^2(J)}+\|T_{A,n,\ve}\chi_I\|_{L^2(I\setminus{J})}.
				\end{equation*}
				For the second term, note that if $x\in I\setminus{J}$, then $x\not\in [-1,2]$ and $A(x)=0$. Moreover, if $y\not\in S$, then $A(y)=0$. Therefore,
				\begin{equation*}
					|T_{A,n,\ve}\chi_I(x)| = \bigg\rvert \int_{\substack{I\cap[0,1] \\ |x-y|>\ve}} \left( \frac{-A(y)}{x-y} \right)^n \frac{dy}{x-y} \bigg\rvert \leq M^n \int_{I\cap[0,1]} \frac{dy}{|x-y|} \leq C M^n,
				\end{equation*}
				and we deduce $\|T_{A,n,\ve}\chi_I\|_{L^2(I\setminus{J})}\leq |I|^{1/2}$. For the first term, for each $x\in J$ we consider the difference
				\begin{equation*}
					T_{A,n,\ve}\chi_I (x) - T_{A,n,\ve}\chi_{J}(x) =  \int_{\substack{I\setminus{J} \\ |x-y|>\ve}} \left( \frac{A(x)-A(y)}{x-y} \right)^n \frac{dy}{x-y} = \int_{\substack{I\setminus{S_1} \\ |x-y|>\ve}} \left( \frac{A(x)}{x-y} \right)^n \frac{dy}{x-y}.
				\end{equation*}
				If $x\in J\setminus{S}$, then $A(x)=0$ and the previous difference vanishes. If $x\in J\cap S$, then $|x-y|>1$ and we can bound it explicitly:
				\begin{align*}
					|T_{A,n,\ve}\chi_I (x) - T_{A,n,\ve}\chi_{J}(x)| &\leq \int_{I\setminus{S_1}} \bigg\rvert \frac{A(x)}{x-y} \bigg\rvert^n \frac{dy}{|x-y|}\\
					&\leq \left( \int_{I\setminus{S_1}} \bigg\rvert \frac{A(x)}{x-y} \bigg\rvert^{2n} dy \right)^{\frac{1}{2}}\left(\int_{\R\setminus{S_1}}\frac{dy}{|x-y|^2}\right)^{\frac{1}{2}} \leq C |I|^{1/2}.
				\end{align*}
				Hence,
				\begin{equation*}
					\| T_{A,n,\ve}\chi_I - T_{A,n,\ve}\chi_{J} \|_{L^2(J)} \leq C |I|^{1/2}|J\cap S|^{1/2}\leq C |I|^{1/2}.
				\end{equation*}
				To sum up: if $I\subset S_1$, then $I=J$ and $\|T_{A,n,\ve}\chi_I\|_{L^2(I)} = \|T_{A,n,\ve}\chi_{J}\|_{L^2(J)}$. If $I \not\subset S_1$, the assumption $I\cap[0,1]\neq \varnothing$ implies $|I|>1$. Therefore
				\begin{equation*}
					\|T_{A,n,\ve}\chi_I\|_{L^2(I)} \leq \|T_{A,n,\ve}\chi_{J}\|_{L^2(J)} +  C |I|^{1/2},
				\end{equation*}
				meaning
				\begin{equation*}
					\frac{1}{|I|^{1/2}} \|T_{A,n,\ve}\chi_I\|_{L^2(I)} \leq \frac{1}{|J|^{1/2}} \|T_{A,n,\ve}\chi_{J}\|_{L^2(J)} + C,
				\end{equation*}
				and the result follows.\end{proof}
			Fix an interval $J:=[p_0,p_1]\subset [-1,2]$. 
			By \eqref{eq:unweighted-testing-condition} and Lemma~\ref{Wn}, the $L^2$ testing condition is reduced to estimating
			\begin{equation}
				\label{symmetrized-integral}
				\iiint_{\Sigma_\ve}
				\left(
				\frac{\alpha-\beta}{z-x}
				\right)^2
				W_n(\alpha,\beta,\gamma)
				\,dx\,dy\,dz.
			\end{equation}
			Here we recall that $\Sigma_\ve:=\{(x,y,z)\in J^3 \,:\, x<y<z,\, y-x>\ve,\, z-y>\ve\}$ and
			\begin{equation*}
				\alpha:=\frac{A(y)-A(x)}{y-x}, \quad \beta:=\frac{A(z)-A(y)}{z-y}, \quad
				\gamma:=\frac{A(z)-A(x)}{z-x},
			\end{equation*}
			that are such that $|\alpha|,|\beta|,|\gamma|\leq 1$, since $\|A'\|_\infty=1$. Set
			\begin{equation*}
				h:=z-x \qquad \text{and} \qquad \lambda:=\frac{y-x}{z-x},
			\end{equation*}
			so that $z=x+h$ and $y=x+\lambda h$. On $\Sigma_\ve$ we have $h\in[2\ve,|J|]$. Also, for each fixed $h$ we have $x\in[p_0,p_1-h]$ and for each fixed $h$ and $x$ we also have $\lambda\in(\frac{\ve}{h},1-\frac{\ve}{h})$. Therefore,  the change of variables $(x,y,z)\mapsto (x,x+\lambda h, x+h)$, whose Jacobian equals $|-h|=h$, transforms the triple integral into
			\begin{equation}
				\label{eq1.1}
				\int_{2\ve}^{|J|}\int_{p_0}^{p_1-h}\left(\int_{\frac{\ve}{h}}^{1-\frac{\ve}{h}} (\alpha-\beta)^2 W_n(\alpha,\beta,\gamma) d \lambda \right) d x \frac{d h}{h}.
			\end{equation}
			Observe that now
			\begin{equation*}
				\alpha(x,\lambda,h):=\frac{A(x+\lambda h)-A(x)}{\lambda h}, 
				\qquad 
				\beta(x,\lambda,h):=\frac{A(x+h)-A(x+\lambda h)}{(1-\lambda)h},
			\end{equation*}
			and these quotients can be rewritten in integral form as
			$$
			\alpha(x,\lambda,h)
			= \frac{1}{\lambda}\int_0^\lambda A'(x+hu)\,d u,\qquad
			\beta(x,\lambda,h)
			= \frac{1}{1-\lambda}\int_\lambda^1 A'(x+hu)\,d u.
			$$
			Therefore, if $\omega$ denotes the modulus of continuity of $A'$, then, for each triple $(x,\lambda,h)$ we have
			\begin{equation}\label{elem}
				|\alpha(x,\lambda,h)-\beta(x,\lambda,h)|\leq 2\omega(h).
			\end{equation}
			
			Indeed, for all $u,v\in(0,1)$ we have
			$$
			\left|A'(x+hu)-A'(x+hv)\right|\le \omega(h|u-v|),
			$$
			and \eqref{elem} follows immediately.
			
			Now we use the identity
			\begin{align*}
				(\alpha-\beta)^2 W_n(\alpha,\beta,\gamma)
				=
				(\alpha-\beta)\left(
				\alpha^n \sum_{k=0}^{n-1}\gamma^{n-1-k}\beta^k
				-
				\beta^n \sum_{k=0}^{n-1}\gamma^{n-1-k}\alpha^k
				\right).
			\end{align*}
			Since $|\alpha|,|\beta|,|\gamma|\le 1$, it follows from \eqref{elem} that
			\begin{align*}
				(\alpha-\beta)^2|W_n(\alpha,\beta,\gamma)|
				&\le
				|\alpha-\beta|
				\left(
				\left|
				\sum_{k=0}^{n-1}\gamma^{n-1-k}\beta^k
				\right|
				+
				\left|
				\sum_{k=0}^{n-1}\gamma^{n-1-k}\alpha^k
				\right|
				\right) \\
				&\le 2n|\alpha-\beta|
				\le 4n\,\omega(h).
			\end{align*}
			Returning to \eqref{eq1.1}, we obtain the bound
			\begin{align*}
				\int_{2\ve}^{|J|}\int_{p_0}^{p_1-h}
				\left(
				\int_{\frac{\ve}{h}}^{1-\frac{\ve}{h}} 4n\,\omega(h)\,d\lambda
				\right)
				d x\,\frac{d h}{h}
				\le Cn\,|J| \int_0^1 \frac{\omega(h)}{h}\,d h.
			\end{align*}
			
			From this estimate we conclude that there exists an absolute constant $C>0$ such that
			\begin{equation*}
				\frac{\|T_{A,n,\ve}\chi_J\|_{L^2(J)}}{|J|^{1/2}}
				\le C\bigl(1+\sqrt{nC_\omega}\bigr),
				\qquad
				C_\omega:=\int_0^1 \frac{\omega(h)}{h}\,d h.
			\end{equation*}
			
			Combining this estimate with Lemma~\ref{lem1.1}, we obtain
			$$
			\Lambda_n
			\lesssim
			1+\sqrt{nC_\omega}.
			$$
			Lemma~\ref{propo} therefore yields
			$$
			\|T_{A,n}\|_{L^2\to L^2}
			\lesssim
			1+\sqrt{nC_\omega}+\log n,
			$$
			which proves Theorem~\ref{2}.
			
			\section{Proof of Lemma~\ref{3}: reduction to a half-derivative estimate}
			\label{secciomig}
			
			The main goal of this section is to prove Lemma~\ref{3}, which provides the key structural reduction used later in the paper. It reduces the estimate of the symmetrized integral \eqref{symmetrized-integral} to the control of a half-order fractional derivative of the auxiliary function
			\begin{equation}\label{Fn}
				F_n(t,x):= \left(\frac{A(x+t)-A(x)}{t}\right)^n.
			\end{equation}
			Observe that since $\|A'\|_\infty=1$, we have $|F_n|\leq 1$. More precisely, we prove that 
			\begin{equation*}
				\|T_{A,n}\|_{L^2\to L^2} \leq C\Big(\sqrt{n}+\|\partial_t^{\frac{1}{2}} F_n\|_{L^2} \Big),
			\end{equation*} for some absolute constant $C$. 
			As an application, we obtain an improved bound for the $L^2$ norm of $T_{A,n}$, with $\sqrt n$ growth under the Sobolev assumption $A'\in H^{1/2}(\mathbb R)$. In the next section, we use Lemma~\ref{3} to derive a more general Besov-space estimate, yielding the same $\sqrt n$ growth.
			
			Recall that, by \eqref{eq:unweighted-testing-condition} and Lemma~\ref{Wn}, controlling the $L^2$ testing condition amounts to estimating the symmetrized integral \eqref{symmetrized-integral}, namely
			\begin{equation*}
				\iiint_{\Sigma_\ve}
				\left(
				\frac{\alpha-\beta}{z-x}
				\right)^2
				W_n(\alpha,\beta,\gamma)
				\,dx\,dy\,dz.
			\end{equation*}
			
			We start by reducing the study of this integral to simpler expressions. We begin with two auxiliary estimates.
			\begin{lem}
				If $\alpha,\;\beta\in [-1,1]$ and $\alpha\beta<0$, then
				\begin{equation*}
					|W_n(\alpha,\beta,\gamma)|\leq \min\bigg\{n,\frac{1}{1-\rho}\bigg\}, \qquad \text{where} \qquad \rho:=\min\{|\alpha|,|\beta|\}.
				\end{equation*}
			\end{lem}
			\begin{proof}
				Assume $\alpha>0$ and $\beta<0$. We first claim that, for every integer $m\geq 1$,
				\begin{equation*}
					|\alpha^m-\beta^m|\leq \alpha-\beta.
				\end{equation*}
				Indeed, if $m$ is even,
				\begin{align*}
					|\alpha^m-\beta^m|=|\alpha^m-|\beta|^m|\leq \max\{\alpha^m,|\beta|^m\}\leq \max\{\alpha,|\beta|\}\leq \alpha+|\beta|=\alpha-\beta.
				\end{align*}
				On the other hand, if $m$ is odd,
				\begin{align*}
					|\alpha^m-\beta^m|=\alpha^m+|\beta|^m\leq \alpha+|\beta|=\alpha-\beta.
				\end{align*}
				So in any case, using the representation
				\begin{equation*}
					W_n(\alpha,\beta,\gamma)=\sum_{k=0}^{n-1} \gamma^k \sum_{j=0}^k \alpha^{n-1-j}\beta^{n-1-k+j} = \sum_{k=0}^{n-1}\gamma^k(\alpha\beta)^{n-1-k}\bigg( \frac{\alpha^{k+1}-\beta^{k+1}}{\alpha-\beta} \bigg),
				\end{equation*}
				we obtain
				\begin{equation*}
					|W_n(\alpha,\beta,\gamma)|\leq \sum_{k=0}^{n-1}|\gamma|^k|\alpha\beta|^{n-1-k}\leq \sum_{k=0}^{n-1} \rho^k \leq \min\bigg\{ n, \frac{1}{1-\rho}\bigg\}.
			\end{equation*}	\end{proof}
			This implies that in the subdomain of integration $\Sigma_\ve \cap\{\alpha\beta<0\}$, repeating the arguments in Section~\ref{secciolipschitz} based on the Fourier transform, the following bound holds:
			\begin{equation*}
				\iiint_{\Sigma_\ve \cap\{\alpha\beta<0\}} \left( \frac{\alpha-\beta}{z-x} \right)^2 W_n(\alpha,\beta,\gamma) d x d y d z \lesssim n\int_{I}|A'(t)-\langle A' \rangle_I|^2 dt.
			\end{equation*}
			This contribution yields a bound for $\|T_{A,n}\|_{L^2\to L^2}$ of the form
			\begin{equation*}
				1+\sqrt{n}\|A'\|_{{\mathrm{BMO}}} \lesssim \sqrt{n},
			\end{equation*}
			already improving the dependence on $n$. Thus, we focus our attention on the subdomain $\Sigma_\ve \cap\{\alpha\beta>0\}$. We split it into the four regions:
			\begin{equation*}
				\Sigma_\ve \cap \Big( \{\alpha>\beta>0\} \cup \{\beta>\alpha>0\} \cup \{\alpha<\beta<0\} \cup  \{\beta<\alpha<0\} \Big),
			\end{equation*}
			which we denote by $\Sigma_\alpha^+, \Sigma_\beta^+, \Sigma_\alpha^-$ and $\Sigma_\beta^-$ respectively. Observe that if $\operatorname{sgn}(\alpha)=\operatorname{sgn}(\beta)=s$, then also
			$\operatorname{sgn}(\gamma)=s$, since $\gamma$ is a convex combination of
			$\alpha$ and $\beta$. Moreover, each monomial appearing in $W_n$ has total
			degree $2n-2$, and hence
			$$
			W_n(\alpha,\beta,\gamma)
			=
			W_n(|\alpha|,|\beta|,|\gamma|).
			$$
			Therefore, it suffices to study the region $\Sigma_\alpha^+$, for example.
			\begin{lem}\label{pos}
				If $\alpha\beta>0$, then
				\begin{equation*}
					(\alpha-\beta)^2W_n(\alpha,\beta,\gamma)\leq (\alpha^n-\beta^n)^2.
				\end{equation*}
			\end{lem}
			\begin{proof}
				Assume $\alpha>0,\beta>0$. By the proof of Lemma \ref{Wn}
				\begin{equation*}
					(\alpha-\beta)^2W_n(\alpha,\beta,\gamma) = \frac{\alpha^n\gamma^n}{\lambda}
					-
					\frac{\alpha^n\beta^n}{\lambda(1-\lambda)}
					+
					\frac{\beta^n\gamma^n}{1-\lambda}.
				\end{equation*}
				Therefore,
				\begin{align*}
					(\alpha^n-\beta^n)^2
					-
					(\alpha-\beta)^2W_n(\alpha,\beta,\gamma)
					&=
					\alpha^{2n}+\beta^{2n}
					+\frac{\alpha^n\beta^n}{\lambda(1-\lambda)}
					-\frac{\alpha^n\gamma^n}{\lambda}
					-\frac{\beta^n\gamma^n}{1-\lambda} \\
					&=
					\Bigl[
					\lambda\alpha^n+(1-\lambda)\beta^n-\gamma^n
					\Bigr]
					\left(
					\frac{\alpha^n}{\lambda}
					+
					\frac{\beta^n}{1-\lambda}
					\right).
				\end{align*}
				Since $x\mapsto x^n$ is convex on $[0,\infty)$, we have
				$$
				\gamma^n
				=
				(\lambda\alpha+(1-\lambda)\beta)^n
				\le
				\lambda\alpha^n+(1-\lambda)\beta^n.
				$$
				Hence
				$$
				(\alpha^n-\beta^n)^2
				-
				(\alpha-\beta)^2W_n(\alpha,\beta,\gamma)
				\ge0,
				$$
				which proves the claim.\end{proof}
			
			We first observe that $
			\gamma=\lambda\alpha+(1-\lambda)\beta,
			$
			so that
			$
			\alpha-\gamma=(1-\lambda)(\alpha-\beta).
			$
			Since also
			$
			z-y=(1-\lambda)(z-x),
			$
			it follows that
			\begin{equation*}
				\frac{\alpha-\beta}{z-x}
				=
				\frac{\alpha-\gamma}{z-y}.
			\end{equation*}
			
			Hence, by Lemma~\ref{pos}, on $\Sigma_\alpha^+$ we have $\alpha\ge\gamma\ge\beta>0$, and therefore
			\begin{align*}
				(\alpha-\beta)^2W_n(\alpha,\beta,\gamma) &\leq (\alpha^n-\beta^n)^2 = (\alpha^n-\gamma^n)^2\bigg( \frac{\alpha^n-\beta^n}{\alpha^n-\gamma^n} \bigg)^2 \\
				&= (\alpha^n-\gamma^n)^2\bigg(\frac{\alpha-\beta}{\alpha-\gamma}\bigg)^2\Bigg[\frac{\sum_{j=0}^{n-1}\beta^j\alpha^{n-1-j}}{\sum_{j=0}^{n-1}\gamma^j\alpha^{n-1-j}}\Bigg]^2 \leq (\alpha^n-\gamma^n)^2\bigg(\frac{\alpha-\beta}{\alpha-\gamma}\bigg)^2.
			\end{align*}
			Therefore, the study of the integral \eqref{symmetrized-integral} can be reduced to estimating
			\begin{equation*}
				\iiint_{\Sigma^+_\alpha} \bigg( \frac{\gamma^n-\alpha^n}{z-y} \bigg)^2 d x d y d z.
			\end{equation*}
			
			We introduce the variables $t=z-x$ and $s:=y-x$, as well as the auxiliary function
			\begin{equation*}
				F_n(u,v):= \left(\frac{A(v+u)-A(v)}{u} \right)^n.
			\end{equation*}
			With this notation, the previous integral can be rewritten as
			\begin{align*}
				\iiint_{\Sigma^+_\alpha} \frac{1}{(z-y)^2} \bigg[ \bigg( \frac{A(z)-A(x)}{z-x} \bigg)^n&-\bigg( \frac{A(y)-A(x)}{y-x} \bigg)^n \bigg]^2 d x d y d z\\
				&=\iiint_{\Sigma^+_\alpha} \frac{|F_n(t,x)-F_n(s,x)|^2}{|t-s|^2}d x d t d s.
			\end{align*}
			We bound the latter expression as follows:
			\begin{align*}
				\iiint_{\Sigma^+_\alpha} \frac{|F_n(t,x)-F_n(s,x)|^2}{|t-s|^2}&d x d t d s \leq \iint_{\R^2} \frac{\|F_n(t,\cdot)-F_n(s,\cdot)\|_{L^2_x}^2}{|t-s|^2}d t d s\\
				&\hspace{-1cm}=\int_{\R}\frac{1}{|h|^2}\bigg( \int_{\R} \|F_n(s+h,\cdot)-F_n(s,\cdot)\|_{L^2_x}^2 d s \bigg)d h.
			\end{align*}
			Define $f_n:\mathbb R\to L^2(\mathbb R)$ by $f_n(t):=F_n(t,\cdot)$. Notice that the quantity
			\begin{equation*}
				\|f_n\|_{L^2(\R; L^2(\R))}^2:=\int_{\R} \|f_n(t)\|_{L^2_x}^2 d t ,
			\end{equation*}
			is finite. This is easily checked by considering separately the cases $|t|<1$ and $|t|>1$, and using that the support of $A$ is contained in $[0,1]$. This proves, in particular, that $F_n\in L^2(\R^2)$. The space $L^2(\R; L^2(\R))$ of $L^2(\R)$-valued functions is a Hilbert space, so applying Plancherel we deduce:
			\begin{align*}
				\int_{\R}\frac{1}{|h|^2}
				\bigg(
				\int_{\R}
				\|F_n(s+h,\cdot)&-F_n(s,\cdot)\|_{L^2_x}^2\,ds
				\bigg)dh
				=
				\int_{\R}\frac{1}{|h|^2}
				\left(
				\int_{\R}
				\|f_n(s+h)-f_n(s)\|_{L^2_x}^2\,ds
				\right)dh \\
				&=
				\int_{\R}\frac{1}{|h|^2}
				\left(
				\int_{\R}
				|e^{2\pi i h\tau}-1|^2
				\|\widehat f_n(\tau)\|_{L^2_x}^2\,d\tau
				\right)dh \\
				&=
				\int_{\R}
				\left(
				\int_{\R}
				\frac{|e^{2\pi i h\tau}-1|^2}{|h|^2}\,dh
				\right)
				\|\widehat f_n(\tau)\|_{L^2_x}^2\,d\tau =
				C\int_{\R}
				|\tau|\,
				\|\widehat f_n(\tau)\|_{L^2_x}^2\,d\tau \\
				&=
				C\int_{\R}
				\|\partial_t^{1/2}F_n(t,\cdot)\|_{L^2_x}^2\,dt 
				=
				C\|\partial_t^{1/2}F_n\|_{L^2}^2.
			\end{align*}
			where $C>0$ is an absolute constant and, for each fixed $x\in\mathbb R$,
			the operator $\partial_t^{1/2}$ is defined through the Fourier multiplier
			$|\tau|^{1/2}$ in the $t$ variable. This proves Lemma~\ref{3}.
			
			Although the next section yields a more general result in terms of Besov norms, here we derive a simpler estimate under the assumption that $A' \in H^{1/2}(\mathbb R)$. This bound is independent of the Dini estimate in Theorem~\ref{2}. 
			
			Before stating the result, let us recall the definition of $H^{1/2}(\mathbb R)$, the fractional Sobolev space endowed with the norm
			\begin{equation}
				\label{def6.2}
				\|f\|_{H^{1/2}}
				:=
				\bigg(
				\int_{\mathbb R} (1+|\xi|^2)^{1/2} |\widehat{f}(\xi)|^2 \, d\xi
				\bigg)^{1/2},
				\qquad f \in L^2(\mathbb R).
			\end{equation}
			
			The result is as follows.
			
			\begin{cor}
				\label{4}
				If $A$ is a $1$-Lipschitz function supported in $[0,1]$ with
				$A' \in H^{1/2}(\mathbb R)$, then
				$$
				\|T_{A,n}\|_{L^2 \to L^2}
				\le
				C\sqrt{n}\left(1+\|A'\|_{H^{1/2}}\right)^{1/2},
				$$ where $C>0$ is an absolute constant.
			\end{cor}
			
			To prove Corollary \ref{4}, we need the following auxiliary estimate:
			\begin{lem}
				\label{lem3.1}
				There exists an absolute constant $C>0$ such that
				\begin{equation*}
					\|F_n\|_{L^2} \leq C\|A'\|_{L^{2n}}^n\lesssim 1, \qquad \|\partial_t F_n\|_{L^2} \leq C\, n\|A'\|_{H^{1/2}}.
				\end{equation*} 
			\end{lem}
			
			Assuming Lemma~\ref{lem3.1}, we can now prove Corollary~\ref{4}. Since
			\[
			\|\partial_t^{1/2}F_n\|^2_{L^2}
			\simeq
			\int_{\R}|\tau|\|\widehat{f}_n(\tau)\|^2_{L^2_x}\,d\tau,
			\]
			Cauchy--Schwarz and Plancherel's identity give
			\begin{align*}
				\|\partial_t^{1/2}F_n\|^2_{L^2}
				&\lesssim
				\left(
				\int_{\R}\|\widehat{f}_n(\tau)\|_{L^2_x}^2\,d\tau
				\right)^{1/2}
				\left(
				\int_{\R}|\tau|^2\|\widehat{f}_n(\tau)\|_{L^2_x}^2\,d\tau
				\right)^{1/2} \\
				&\simeq
				\|F_n\|_{L^2}\|\partial_tF_n\|_{L^2} \lesssim
				n\|A'\|_{H^{1/2}}.
			\end{align*}
			Thus
			$
			\|\partial_t^{1/2}F_n\|_{L^2}
			\lesssim
			\sqrt{n
				\|A'\|_{H^{1/2}}
			}.
			$
			Combining this estimate with Lemma~\ref{3}, we obtain 
			$$
			\|T_{A,n}\|_{L^2\to L^2}
			\le	C\sqrt{n}\left(1+\|A'\|_{H^{1/2}}\right)^{1/2}.
			$$
			This proves Corollary~\ref{4}.
			
			\begin{proof}[Proof of Lemma~\ref{lem3.1}]
				Define
				\begin{equation*}
					G(t,x):=\frac{A(x+t)-A(x)}{t}=\int_0^1 A'(x+\theta t)d\theta,
				\end{equation*}
				so that $F_n(t,x)=G^n(t,x)$. Since $|A'|\le 1$, we have $|G(t,x)|\le 1$. We first estimate $G$ in $L_x^p$.  For any $1\leq p <\infty$ , Minkowski's integral inequality yields
				\begin{equation*}
					\|G(t,\cdot)\|_{L^p_x} = \bigg\| \int_0^1 A'(\cdot +\theta t)d\theta \bigg\|_{L^p_x}\leq  \int_0^1 \| A'(\cdot +\theta t) \|_{L^p_x}d \theta = \|A'\|_{L^p}.
				\end{equation*}
				Since  $A(0)=A(1)=0$, we have for every $x\in[0,1]$,
				\begin{equation*}
					|A(x)|=\bigg\rvert \int_0^x A'(s)d s \bigg\rvert \leq x^{1-\frac{1}{p}}\bigg( \int_{\R}|A'(s)|^pd s \bigg)^{\frac{1}{p}} \leq \|A'\|_{L^p}.
				\end{equation*}
				Hence, for every $t\neq 0$, the following bound also holds:
				\begin{equation*}
					\|G(t,\cdot)\|_{L^p_x}\leq \frac{2}{|t|}\|A\|_{L^p}\leq \frac{2}{|t|}\|A'\|_{L^p}.
				\end{equation*}
				Combining the previous estimates, we obtain
				\begin{equation}
					\label{eq3.1}
					\|G(t,\cdot)\|_{L^p_x}\leq \|A'\|_{L^p} \min\bigg\{ 1,\frac{2}{|t|} \bigg\}.
				\end{equation}
				Choosing $p=2n$, it follows that
				\begin{align*}
					\|F_n\|_{L^2}^2 &= \int_{\R}\bigg( \int_{\R} |G(t,x)|^{2n} d x \bigg) d t \leq \|A'\|_{L^{2n}}^{2n}\int_{\R} \min\bigg\{ 1,\frac{2}{|t|} \bigg\}^{2n}d t\\
					&=\|A'\|_{L^{2n}}^{2n} \bigg( 4+\frac{4}{2n-1} \bigg) \leq 8 \|A'\|_{L^{2n}}^{2n},
				\end{align*}
				which proves the first estimate, since $A'$ is supported in $[0,1]$ and therefore
				$\|A'\|_{L^{2n}}\le 1.$ To estimate $\|\partial_tF_n\|_{L^2}$, we compute
				\begin{align*}
					\|\partial_t F_n\|_{L^2}^2 &= n^2\iint_{\R^2} |G(t,x)|^{2n-2}|\partial_t G(t,x)|^2 d x d t \leq n^2\iint_{\R^2} |\partial_t G(t,x)|^2 d x d t \\
					&= n^2 \iint_{\R^2} \bigg\rvert \frac{A'(x+t)-\frac{A(x+t)-A(x)}{t}}{t} \bigg\rvert^2 d x d t\\
					&= n^2 \int_{\R}\bigg( \int_{\R} \bigg\rvert \bigg( \frac{2\pi i \xi t\, e^{2\pi i \xi t}-e^{2\pi i \xi t}+1}{t^2} \bigg)\widehat{A}(\xi) \bigg\rvert^2 d\xi \bigg)d t\\
					&= 8\pi^3 n^2\bigg( \int_{\R} \frac{|iu\,e^{iu}-e^{iu}+1|^2}{u^4} d u \bigg)\bigg( \int_{\R} |\xi|^3|\widehat{A}(\xi)|^2 d \xi \bigg)\\
					&=\frac{16\pi^4}{3} n^2 \int_{\R} |\xi|^3|\widehat{A}(\xi)|^2 d \xi = Cn^2\int_{\R}|\xi||\widehat{A'}(\xi)|^2d\xi.
				\end{align*}
				The last expression is precisely the squared homogeneous $\dot H^{1/2}$ seminorm of $A'$, where the dot denotes the homogeneous Sobolev space. Since
				$\displaystyle
				\|A'\|_{\dot H^{1/2}} \le \|A'\|_{H^{1/2}},
				$
				the second estimate follows. 
			\end{proof}
			
			\section{Proof of Theorem~\ref{5}: the Besov case}
			\label{secciobesov}
			
			To prove Theorem~\ref{5} in the case $\|A'\|_\infty=1$, namely, to establish
			$$
			\|T_{A,n}\|_{L^2\to L^2}
			\le
			C\sqrt n\left(
			1+\|A'\|_{B^{0,1}_{1,1}}
			\right)^{1/2},
			$$
			we use the reduction provided by Lemma~\ref{3}. Thus, it remains to estimate
			$$
			\|\partial_t^{1/2}F_n\|_{L^2}^2
			\simeq
			\int_{\mathbb R} |\tau|\,\|\widehat f_n(\tau)\|_{L_x^2}^2\,d\tau,
			$$
			where $F_n$ is the function defined in \eqref{Fn}. Write
			$$\displaystyle
			G(t,x):=\frac{A(x+t)-A(x)}{t}.
			$$
			Then $F_n=G^n$.
			Let us fix $\psi\in \mathcal{C}_c^\infty(\R)$ with $0\leq \psi \leq 1$ and so that $\text{supp}(\psi)\subset [-3,3]$ and $\psi\equiv 1$ in $[-2,2]$. We have
			\begin{equation*}
				\|\partial_t^{\frac{1}{2}} F_n\|_{L^2}^2\leq 2\|\partial_t^{\frac{1}{2}} [(1-\psi(t))F_n]\|_{L^2}^2+2\|\partial_t^{\frac{1}{2}} [\psi(t)F_n]\|_{L^2}^2.
			\end{equation*}
			We begin with the first term. By \eqref{eq3.1}
			\begin{align*}
				\|F_n(t,\cdot)\|_{L^2_x} = \|G(t,\cdot)\|_{L^{2n}_x}^n \leq \frac{2^n}{|t|^n}\|A'\|_{L^{2n}}^n\leq \frac{2^n}{|t|^n}.
			\end{align*}
			Then,
			$\displaystyle
			\|F_n(t,\cdot)\|_{L^2_x}\leq 1$ if $|t|>2.$
			Also, since $A(0)=A(1)=0,$ it is not hard to prove that for $u\in [0,1]$, we have $|A(u)|\leq \min{(u,1-u)} \leq \frac{1}{2}$. Hence,  for every $t\neq 0$,
			\begin{align*}
				|G(t,x)|= \bigg\rvert \frac{A(x+t)-A(x)}{t} \bigg\rvert \leq \frac{1}{|t|}.
			\end{align*}
			For almost every $|t|>1$, we also have
			\begin{equation}\label{cota}
				|\partial_t G(t,x)|
				=
				\left|
				\frac{tA'(x+t)-A(x+t)+A(x)}{t^2}
				\right|
				\lesssim
				\frac1{|t|}.
			\end{equation}
			%and it also holds for $|t|\neq 0$,
			%\begin{align*}
			%	|\partial_t G(t,x)|= \frac{1}{M}\bigg\rvert \frac{tA'(x+t)-A(x+t)+A(x)}{t^2} \bigg\rvert \leq \frac{2}{|t|}.
			%\end{align*}
			Since for each fixed $t$ we have $|\text{supp}_x(G(t,\cdot))|\leq 2$, we deduce that for almost every $|t|>1$,
			\begin{align*}
				\|\partial_t F_n(t, \cdot)\|^2_{L^2_x} &= n^2\|G(t,\cdot)^{n-1}\partial_t G(t,\cdot)\|^2_{L^2_x}\\
				&\leq n^2|\text{supp}_x(G(t,\cdot))|\|G(t,\cdot)\|_\infty^{2n-2}\|\partial_t G(t,\cdot)\|_\infty^2\lesssim \frac{n^2}{t^{2n}}.
			\end{align*}
			Since $\operatorname{supp}\psi'\subset\{2\le |t|\le 3\}$ and
			$\|F_n(t,\cdot)\|_{L^2_x}\le 1$ for $|t|>2$, we have
			$\displaystyle
			\|\psi'(t)F_n\|_{L^2}\le \|\psi'\|_{L^2}.
			$
			
			Therefore, applying Plancherel and Cauchy--Schwarz as in the previous section, we obtain a sharper bound for this first term:
			\begin{align*}
				\|\partial_t^{1/2}[(1-\psi(t))F_n]\|_{L^2}^2
				&\leq C \|(1-\psi(t))F_n\|_{L^2}
				\|\partial_t[(1-\psi(t))F_n]\|_{L^2}\\
				&\leq C\bigg(\int_{|t|>2} \|F_n(t,\cdot)\|^2_{L^2_x}\,dt\bigg)^{1/2}
				\Big(
				\|(1-\psi(t))\partial_t F_n\|_{L^2}
				+
				\|\psi'(t)F_n\|_{L^2}
				\Big)\\
				&\lesssim \bigg(\int_{|t|>2} \frac{1}{t^{2n}}\,dt\bigg)^{1/2}
				\bigg[
				n\bigg(\int_{|t|>2} \frac{dt}{t^{2n}} \bigg)^{1/2}
				+
				\|\psi'\|_{L^2}
				\bigg]
				\lesssim 1.
			\end{align*}
			
			Thus, it remains to estimate
			$$
			\|\partial_t^{1/2}(\psi(t) F_n)\|_{L^2}^2.
			$$
			
			For notational convenience, set
			$\displaystyle
			H(t,x):=\psi(t)F_n(t,x).
			$
			Let us first observe that, on the one hand, by Plancherel's identity
			\begin{equation*}
				\|\partial_t^{\frac{1}{2}} H\|_{L^2}^2 = \int_{\R}\int_{\R} |\tau||\widehat{H}(\tau,x)|^2d \tau d x.
			\end{equation*}
			On the other hand, for each fixed $x$, 
			\begin{align*}
				\iint_{\R^2} \frac{|H(t,x)-H(s,x)|^2}{|t-s|^2}d t d s & = \int_{\R}\left( \int_{\R} \big\rvert H(t,x)-H(t-h,x) \big\rvert^2 d t \right) \frac{d h}{h^2}\\
				&=\int_{\R}\left( \int_{\R} \frac{|1-e^{-2\pi i h\tau}|^2}{|h|^2} d h \right)|\widehat{H}(\tau,x)|^2d \tau\\
				&\simeq \int_{\R}\left( \int_{\R} \frac{\sin{(u)}^2}{u^2} d u \right)|\tau||\widehat{H}(\tau,x)|^2d \tau\simeq \int_{\R}|\tau||\widehat{H}(\tau,x)|^2d \tau.
			\end{align*}
			Hence, we obtain the identity
			\begin{equation}\label{identity}
				\|\partial_t^{1/2} H\|_{L^2}^2
				=
				\frac{1}{4\pi^2}
				\int \!\! \iint
				\frac{|H(t,x)-H(s,x)|^2}{|t-s|^2}
				\,dt\,ds\,dx.
			\end{equation}
			We split the integrand as follows:
			\begin{align*}
				\int \!\! \iint
				\frac{|H(t,x)-H(s,x)|^2}{|t-s|^2}
				\,dt\,ds\,dx
				&\leq
				2\int \!\! \iint
				|\psi(t)|^2
				\frac{|G(t,x)^n-G(s,x)^n|^2}{|t-s|^2}
				\,dt\,ds\,dx \\
				&\quad+
				2\int \!\! \iint
				|\psi(t)-\psi(s)|^2
				\frac{|G(s,x)|^{2n}}{|t-s|^2}
				\,dt\,ds\,dx \\
				&=: 2I_1+2I_2.
			\end{align*}
			Let us first estimate $I_2$. We rewrite it as
			\begin{equation*}
				\int\!\!\iint
				|\psi(t)-\psi(s)|^2
				\frac{|G(s,x)|^{2n}}{|t-s|^2}
				\,dt\,ds\,dx
				=
				\int
				\bigg(
				\int
				\frac{|\psi(t)-\psi(s)|^2}{|t-s|^2}
				\,dt
				\bigg)
				\|G(s,\cdot)\|_{L_x^{2n}}^{2n}
				\,ds.
			\end{equation*}
			For the inner integral, observe that if $|s|\leq 6$
			\begin{align*}
				\int_{\R}\frac{ |\psi(t)-\psi(s)|^2}{|t-s|^2}d t &= \int_{|t-s|\leq 1}\frac{ |\psi(t)-\psi(s)|^2}{|t-s|^2}d t+\int_{|t-s|> 1}\frac{ |\psi(t)-\psi(s)|^2}{|t-s|^2}d t\\
				&\leq 2\|\psi'\|^2_\infty+\int_{|t-s|> 1}\frac{ 4\|\psi\|^2_\infty}{|t-s|^2}d t \lesssim 1.
			\end{align*}
			Also, if $|s|>6$,
			\begin{align*}
				\int_{\R}\frac{ |\psi(t)-\psi(s)|^2}{|t-s|^2}d t = \int_{|t|<3}\frac{ |\psi(t)|^2}{|t-s|^2}d t \leq \frac{4}{s^2} \int_{|t|<3}|\psi(t)|^2d t \lesssim \frac{1}{s^2}.
			\end{align*}
			Therefore we deduce
			\begin{equation}
				\int_{\R}\frac{ |\psi(t)-\psi(s)|^2}{|t-s|^2}d t \lesssim \frac{1}{1+s^2}, \qquad \forall s \in \R,
			\end{equation}
			which implies, 
			\begin{align*}
				I_2=\int_{\R}\bigg(\int_{\R}\frac{ |\psi(t)-\psi(s)|^2}{|t-s|^2}&d t \bigg) \|G(s,\cdot)\|_{L_x^{2n}}^{2n} d s \lesssim \int_{\R} \frac{\|G(s,\cdot)\|_{L^{2n}_x}^{2n}}{1+s^2}d s \leq \int_{\R} \frac{d s}{1+s^2}\lesssim 1,
			\end{align*}
			because $|G|\leq 1$ and $|\mbox{supp}_x(G(s,\cdot))|\leq 2$.
			
			So we are left with $I_1$.
			%\begin{equation}
			%	\label{eq4.2}
			%	\int_{\R}\iint_{\R^2} \psi(t)^2\frac{|G(t,x)^n-G(s,x)^n|^2}{|t-s|^2}d t d s d x.
			%\end{equation}
			For this estimate, we use two lemmas. First, we introduce the operator
			\begin{equation*}
				\Lambda f(t)
				:=
				c\,\mathrm{p.v.}\int_{\mathbb R}
				\frac{f(t)-f(s)}{|t-s|^2}\,ds,
			\end{equation*}
			where the constant \(c>0\) is chosen so that \(\Lambda\) has Fourier symbol \(|\tau|\). The definition can be understood for \(f\in C_c^\infty(\mathbb R)\).
			
			We consider the associated bilinear form: for \(f,g\in C_c^\infty(\mathbb R)\),
			\begin{align*}
				\langle g, \Lambda f \rangle
				&:=
				\int_{\mathbb R} g(t)\Lambda f(t)\,dt  =
				c\lim_{\varepsilon\to0}
				\int_{\mathbb R}
				\left(
				\int_{|t-s|>\varepsilon}
				g(t)\frac{f(t)-f(s)}{|t-s|^2}\,ds
				\right)\,dt
				=: c\lim_{\varepsilon\to0} I_\varepsilon.
			\end{align*}
			Observe that we can rewrite \(I_\varepsilon\) as
			\begin{align*}
				I_\varepsilon
				&=
				\int_{\mathbb R}
				\left(
				\int_{|t-s|>\varepsilon}
				\frac{f(t)g(t)}{|t-s|^2}\,ds
				\right)\,dt
				-
				\int_{\mathbb R}
				\left(
				\int_{|t-s|>\varepsilon}
				\frac{f(s)g(t)}{|t-s|^2}\,ds
				\right)\,dt \\
				&=
				\int_{\mathbb R}
				\left(
				\int_{|t-s|>\varepsilon}
				\frac{f(t)g(t)}{|t-s|^2}\,ds
				\right)\,dt
				-
				\int_{\mathbb R}
				\left(
				\int_{|t-s|>\varepsilon}
				\frac{f(t)g(s)}{|t-s|^2}\,ds
				\right)\,dt \\
				&=
				\int_{\mathbb R}
				\left(
				\int_{|t-s|>\varepsilon}
				f(t)\frac{g(t)-g(s)}{|t-s|^2}\,ds
				\right)\,dt,
			\end{align*}
			where in the second integral we exchanged the roles of \(s\) and \(t\). Adding the two expressions for \(I_\varepsilon\) and dividing by \(2\), we obtain
			\begin{align*}
				\langle g,\Lambda f\rangle
				=
				\frac{c}{2}
				\iint_{\mathbb R^2}
				\frac{(f(t)-f(s))(g(t)-g(s))}{|t-s|^2}\,dt\,ds.
			\end{align*}
			We now establish a localized fractional integration by parts lemma.
			\begin{lem}
				\label{lem4.1}
				Let $\eta\in C_c^\infty$, and let $u$ and $v$ be sufficiently regular functions. Then
				\begin{align*}
					\iint
					\eta(t)\frac{(u(t)-u(s))(v(t)-v(s))}{|t-s|^2}\,dt\,ds
					&=
					\frac{2}{c}\langle \eta v, \Lambda u \rangle \\
					&\quad-
					\iint
					v(t)\frac{(\eta(t)-\eta(s))(u(t)-u(s))}{|t-s|^2}\,dt\,ds.
				\end{align*}
			\end{lem}
			\begin{proof}
				By definition,
				\begin{align*}
					\frac{2}{c}\langle \eta v, \Lambda u \rangle
					&=
					\iint
					\frac{(u(t)-u(s))(\eta(t)v(t)-\eta(s)v(s))}{|t-s|^2}\,dt\,ds \\
					&=
					\iint
					\frac{(u(t)-u(s))\big[\eta(t)(v(t)-v(s))+v(s)(\eta(t)-\eta(s))\big]}{|t-s|^2}\,dt\,ds \\
					&=
					\iint
					\eta(t)\frac{(u(t)-u(s))(v(t)-v(s))}{|t-s|^2}\,dt\,ds \\
					&\quad+
					\iint
					v(s)\frac{(u(t)-u(s))(\eta(t)-\eta(s))}{|t-s|^2}\,dt\,ds.
				\end{align*}
				Exchanging the roles of $t$ and $s$ in the last integral yields the result.
			\end{proof}
			We will also need the following algebraic inequality:
			\begin{lem}
				\label{lem4.2}
				Let $a,b\in \R$. Then,
				$\displaystyle
				(a^n-b^n)^2\leq n(a-b)(a^{2n-1}-b^{2n-1}).
				$
			\end{lem}
			\begin{proof}
				If $a=b$, the claim is immediate. Now observe that
				\begin{equation*}
					a^n-b^n = n(a-b)\int_0^1(\theta a +(1-\theta)b)^{n-1}d \theta.
				\end{equation*}
				Squaring both sides and applying Cauchy--Schwarz inequality, we obtain 
				\begin{align*}
					(a^n-b^n)^2&\leq n^2(a-b)^2\int_0^1(\theta a +(1-\theta)b)^{2n-2}d \theta = n^2(a-b)^2\left[\frac{a^{2n-1}-b^{2n-1}}{(2n-1)(a-b)}\right]\\
					&\leq n(a-b)(a^{2n-1}-b^{2n-1}),
				\end{align*} as claimed.
			\end{proof}
			Returning to $I_1$ and applying Lemmas \ref{lem4.2} and \ref{lem4.1} as well as the fact that $|G|\leq 1$, we get
			\begin{align*}
				0\leq I_1
				&\leq n\int\!\!\iint
				\psi(t)^2
				\frac{(G(t,x)-G(s,x))(G(t,x)^{2n-1}-G(s,x)^{2n-1})}{|t-s|^2}
				\,dt\,ds\,dx \\
				&\lesssim n\bigg[
				\left|
				\int
				\langle \psi^2 G(\cdot,x)^{2n-1},\Lambda_tG(\cdot,x)\rangle\,dx
				\right| \\
				&\hspace{1.6cm}
				+
				\int\left|
				\iint
				G(t,x)^{2n-1}
				\frac{(G(t,x)-G(s,x))(\psi(t)^2-\psi(s)^2)}{|t-s|^2}
				\,dt\,ds
				\right|dx
				\bigg].
			\end{align*}
			Here $\Lambda_t$ denotes the operator $\Lambda$ acting on the $t$-variable. The second term above is not difficult to estimate. Indeed, notice that
			\begin{align}
				\bigg\rvert \iint_{\R^2} G(t,x)^{2n-1}&\frac{(G(t,x)-G(s,x))(\psi(t)^2-\psi(s)^2)}{|t-s|^2}d t d s \bigg\rvert \nonumber \\
				&\leq \int_{\R} |G(t,x)|\bigg( \int_{\R} \frac{|G(t,x)-G(s,x)||\psi(t)^2-\psi(s)^2|}{|t-s|^2}d s  \bigg)d t. \label{eq4.3}
			\end{align}
			The following lemma deals with the inner integral.
			\begin{lem}
				\label{lem4.3}
				Fix $x\in \mathbb{R}$. Then, there exists a constant $C>0$, depending only on $\psi$, such that for every $t\in \mathbb{R}$
				\begin{equation*}
					\int_{\R} \frac{|G(t,x)-G(s,x)||\psi(t)^2-\psi(s)^2|}{|t-s|^2}\,ds \leq \frac{C}{1+t^2}.
				\end{equation*}
			\end{lem}
			
			\begin{proof}
				Set $u(t):=G(t,x)$ and $\phi:=\psi^2$. By definition $|u|\leq 1$ and also recall that, by \eqref{cota}, for $|t|>1$, we have $|u'(t)|\lesssim |t|^{-1}$. This implies, in particular, that if $s,t\in\mathbb R\setminus(-1,1)$ and
				$|s-t|<1$, then
				$
				|u(t)-u(s)|\lesssim |t-s|.
				$
				
				As $\phi\equiv 1$ in $[-2,2]$, we have
				\begin{equation*}
					\phi(2)=1, \quad \phi'(2)=0, \qquad \phi(-2)=1, \quad \phi'(-2)=0.
				\end{equation*} 
				Moreover, since $\operatorname{supp}(\phi)\subset [-3,3]$, we also have $\phi(3)=\phi'(3)=\phi(-3)=\phi'(-3)=0$.
				
				We estimate the integral by considering 4 regions: $|t|\leq 2$, $2\leq |t| \leq 3$, $3\leq |t| \leq 4$ and $|t|>4$. We begin with the first one. Fix $t\in[-2,2]$, so that $\phi(t)=1$. In this case the integral becomes
				\begin{equation*}
					\int_{2<|s|<3} \frac{|1-\phi(s)||u(t)-u(s)|}{|t-s|^2}\,ds
					+
					\int_{|s|>3} \frac{|u(t)-u(s)|}{|t-s|^2}\,ds.
				\end{equation*}
				By Taylor's theorem, $|1-\phi(s)|\lesssim (s-2)^2$ as well as $|1-\phi(s)|\lesssim (s+2)^2$. Therefore, since $|u(t)-u(s)|\leq 2$ we get
				\begin{align*}
					\int_{2<|s|<3} \frac{|1-\phi(s)||u(t)-u(s)|}{|t-s|^2}\,ds\lesssim 1.
				\end{align*}
				Since $|u(t)-u(s)|\leq 2$ and $|t|\leq 2$, we also have
				\begin{align*}
					\int_{|s|>3} \frac{|u(t)-u(s)|}{|t-s|^2}\,ds
					\leq
					2\int_{|s|>3} \frac{ds}{|t-s|^2}
					\lesssim 1.
				\end{align*}
				
				Now fix $2\leq |t| \leq 3$. We assume that $t\in (2,3)$, since the case $t\in (-3,-2)$ is treated analogously. We split the integral as follows:
				\begin{equation*}
					\int_{|s-t|<1}\frac{|\phi(t)-\phi(s)||u(t)-u(s)|}{|t-s|^2}\,ds
					+
					\int_{|s-t|\geq 1}\frac{|\phi(t)-\phi(s)||u(t)-u(s)|}{|t-s|^2}\,ds.
				\end{equation*}
				If $|s-t|<1$ and $t\in (2,3)$, then $s\in (1,4) \subset \R\setminus(-1,1)$. This implies $|u(t)-u(s)|\lesssim|t-s|$. Moreover $
				|\phi(t)-\phi(s)|\lesssim |t-s|$. Therefore
				\begin{equation*}
					\int_{|s-t|<1}\frac{|\phi(t)-\phi(s)||u(t)-u(s)|}{|t-s|^2}\,ds
					\lesssim 1.
				\end{equation*}
				The study of the second term is also straightforward:
				\begin{equation*}
					\int_{|s-t|\geq 1}\frac{|\phi(t)-\phi(s)||u(t)-u(s)|}{|t-s|^2}\,ds
					\lesssim\int_{|s-t|\geq 1} \frac{ds}{|t-s|^2}
					\lesssim 1.
				\end{equation*}
				So in this case the integral is again $\lesssim 1$.
				
				Next fix $3\leq |t|\leq 4$. We assume that $t\in(3,4)$, the negative
				case being analogous. Since $\phi(t)=0$, we write
				$$
				\int_{\mathbb R}
				\frac{|u(t)-u(s)|\,|\phi(t)-\phi(s)|}{|t-s|^2}\,ds
				=
				\int_{-3}^{2}
				\frac{|\phi(s)|\,|u(t)-u(s)|}{|t-s|^2}\,ds
				+
				\int_{2}^{3}
				\frac{|\phi(s)|\,|u(t)-u(s)|}{|t-s|^2}\,ds.
				$$
				The first integral is bounded because its integrand has no singularity.
				For the second one, Taylor's theorem gives
				$$
				|\phi(s)|\lesssim (3-s)^2,\qquad s\in(2,3),
				$$
				and therefore, using $|u(t)-u(s)|\leq 2$,
				$$
				\int_{2}^{3}
				\frac{|\phi(s)|\,|u(t)-u(s)|}{|t-s|^2}\,ds
				\lesssim 1.
				$$
				Thus the integral is $\lesssim 1$ for $3\leq |t|\leq 4$. Combining the previous cases, for $|t|\leq 4$ we have
				$$
				\int_{\mathbb R}
				\frac{|G(t,x)-G(s,x)|\,|\psi(t)^2-\psi(s)^2|}{|t-s|^2}\,ds
				\lesssim 1.
				$$
				Since $|t|\leq 4$, we have $(1+t^2)^{-1}\gtrsim 1$, and hence
				$$
				\int_{\mathbb R}
				\frac{|G(t,x)-G(s,x)|\,|\psi(t)^2-\psi(s)^2|}{|t-s|^2}\,ds
				\lesssim
				\frac{1}{1+t^2}.
				$$
				Finally, let $|t|>4$. Then $\phi(t)=0$ and $\operatorname{supp}\phi\subset[-3,3]$.
				Thus
				$$
				\int_{\mathbb R}
				\frac{|u(t)-u(s)|\,|\phi(t)-\phi(s)|}{|t-s|^2}\,ds
				=
				\int_{-3}^{3}
				\frac{|\phi(s)|\,|u(t)-u(s)|}{|t-s|^2}\,ds.
				$$
				Using $|u|\leq 1$ and $|t-s|\gtrsim |t|$ for $s\in[-3,3]$, we get
				$$
				\int_{-3}^{3}
				\frac{|\phi(s)|\,|u(t)-u(s)|}{|t-s|^2}\,ds
				\lesssim
				\frac{1}{t^2}
				\lesssim
				\frac{1}{1+t^2}.
				$$
				The implicit constants depend only on $\psi$, and the proof follows.
			\end{proof}
			
			Returning to \eqref{eq4.3}, we obtain
			$$
			\left|
			\iint_{\mathbb R^2} G(t,x)^{2n-1}
			\frac{(G(t,x)-G(s,x))(\psi(t)^2-\psi(s)^2)}{|t-s|^2}
			\,dt\,ds
			\right|
			\le
			C\int_{\mathbb R}\frac{|G(t,x)|}{1+t^2}\,dt.
			$$
			Since \(|G(t,x)|\le1\) and, for each fixed \(t\),
			$
			|\operatorname{supp}_x G(t,\cdot)|\le2,
			$
			we have
			$
			\|G(t,\cdot)\|_{L^1_x}\le2.
			$
			Integrating over \(x\), the right-hand side is therefore bounded by
			$$
			\int_{\mathbb R}\int_{\mathbb R}\frac{|G(t,x)|}{1+t^2}\,dt\,dx
			=
			\int_{\mathbb R}\frac{\|G(t,\cdot)\|_{L^1_x}}{1+t^2}\,dt
			\lesssim1.
			$$
			Therefore, all in all, we have obtained
			\begin{equation}
				\label{eq4.4}
				\|\partial_t^{\frac{1}{2}} F_n\|_{L^2}^2 \lesssim n \left[ 1+\bigg\rvert \int_{\R} \langle \psi^2 G(\cdot, x)^{2n-1}, \Lambda_t G(\cdot,x) \rangle d x \bigg\rvert \right] \leq n\left[ 1+\|\psi^2(t) \Lambda_t G\|_{L^1(\R^2)} \right].
			\end{equation}
			
			Thus it remains to estimate $\|\psi^2(t)\Lambda_t G\|_{L^1}$.
			
			To this end, recall that $G(t,x):=\int_0^1 A'(x+\theta t) d \theta$. We use the Littlewood-Paley decomposition 
			$\displaystyle
			A' = \Delta_{-1}A'+\sum_{j\geq 0} \Delta_j A'
			$ defined in \eqref{norm_Besov}.
			We also choose $\widetilde{\varphi}\in \mathcal{C}^\infty_c(\R\setminus{\{0\}})$ supported in $\{\frac{1}{4}\leq |\xi|\leq 4\}$ with $\widetilde{\varphi}|_{\text{supp}(\varphi)}\equiv 1$. We then have
			\begin{equation*}
				G(t,x)=g_{-1}(t,x)+\sum_{j\geq 0} g_j(t,x), \qquad \text{where} \quad g_j(t,x):=\int_0^1 \Delta_j A'(x+\theta t)d \theta.
			\end{equation*}
			For the moment we proceed formally; the estimates below justify the summation. 
			
			Taking the Fourier transform in the $x$-variable, we obtain
			\begin{equation*}
				\widehat{g_j}(t,\xi) = \int_0^1 \widehat{\Delta_j A'}(\xi)e^{2\pi i \theta t \xi} d \theta = \frac{e^{2\pi i t\xi}-1}{2\pi i t\xi}\widehat{\Delta_j A'}(\xi)=: m_0(t\xi) \widehat{\Delta_j A'}(\xi).
			\end{equation*}
			
			Now we apply the operator $\Lambda_t$, with Fourier symbol $|\tau|$, on both sides of the previous identity and obtain
			\begin{align*}
				\widehat{\Lambda_t g_j}(t,\xi)
				=
				\Lambda_t(m_0(t\xi))\widehat{\Delta_jA'}(\xi)
				=
				|\xi|\,(\Lambda m_0)(t\xi)\,\widehat{\Delta_j A'}(\xi)
				=
				|\xi|\,(\Lambda m_0)(t\xi)\,
				\widetilde\varphi(2^{-j}\xi)\widehat{\Delta_j A'}(\xi).
			\end{align*}
			Here the hat denotes the Fourier transform in the $x$ variable, while
			$\Lambda_t$ acts on $m_0(t\xi)$ as a function of $t$. This is justified by
			interpreting the Fourier transforms in the $L^2(\mathbb R)$ sense, in order to show that $\Lambda_t$ and the Fourier transform commute. Let us write $m_{j,t}(\xi):=|\xi|(\Lambda m_0)(t\xi)\widetilde{\varphi}(2^{-j}\xi)$, so that
			\begin{equation*}
				\widehat{\Lambda_t g_j}(t, \xi)=m_{j,t}(\xi)\widehat{\Delta_j A'}(\xi).
			\end{equation*}
			To estimate the associated multiplier norms, we use the following lemma:.
			%\begin{lem}
			%\label{lem4.4}
			%The function $\Lambda m_0$ is smooth and satisfies
			%\begin{equation*}
			%|(\Lambda m_0)^{(k)}(u)|\leq C_k \min{\bigg\{ 1,\frac{1}{|u|} \bigg\}}, \qquad \forall u\in \R, k\in \mathbb{Z}_{\geq 0},
			%\end{equation*}
			%where $C_k$ is an absolute constant that only depends on $k$.
			%\end{lem}
			%\begin{proof}
			%Since
			%$$
			%m_0(u)=\int_0^1 e^{2\pi i u \theta}\,d\theta,
			%$$
			%$\displaystyle\widehat{m_0}(\xi)=\chi_{[0,1]}(\xi)$. Therefore, the Fourier transform of \(\Lambda m_0\) is \(\xi \chi_{[0,1]}(\xi)\), and by the inversion formula we obtain
			%$$
			%\Lambda m_0(u)=\int_0^1 \xi e^{2\pi i u\xi}\,d\xi.
			%$$
			%Differentiating under the integral sign, we see that $\Lambda m_0 \in \mathcal{C}^\infty$. Moreover, for each $k\in \mathbb{Z}_{\geq 0}$ and every $u\in \R$ we have
			%\begin{align*}
			%|(\Lambda m_0)^{(k)}(u)| \leq \int_0^1 (2\pi)^{k}\xi^{k+1}d \xi = \frac{(2\pi)^{k}}{k+2}.
			%\end{align*}
			%On the other hand, integrating by parts yields for each $k\in \mathbb{Z}_{\geq 0}$ and  $u\neq 0$,
			%\begin{align*}
			%(2\pi i)^{k}\int_{0}^1 \xi^{k+1}e^{2\pi i \xi u}d \xi = \frac{(2\pi i)^{k-1}}{u}\bigg[ \xi^{k+1}e^{2\pi i \xi u}\Big\rvert_0^1 - (k+1)\int_0^1 \xi^k e^{2\pi i \xi u}d \xi \bigg].
			%\end{align*}
			%Therefore,
			%\begin{equation*}
			%|(\Lambda m_0)^{(k)}(u)| \leq \frac{(2\pi)^k}{|u|}\bigg[ 1+\int_0^1 (k+1)\xi^{k}d \xi \bigg]\leq 2\frac{(2\pi)^k}{|u|},
			%\end{equation*}
			%so choosing $C_k:=2(2\pi)^k$ we are done.
			%\end{proof}
			\begin{lem}
				\label{lem7.5}
				Let $j\geq -1$. For every $t\in\mathbb R$,
				\[
				\|\Lambda_t g_j(t,\cdot)\|_{L^1_x}
				\le
				C\bigl(1+\min\{2^j,|t|^{-1}\}\bigr)\|\Delta_jA'\|_{L^1}.
				\]
			\end{lem}
			\begin{proof}
				We first treat the case $j\geq0$. Fix $j$ and define the Fourier multiplier operator $T_{m_{j,t}}$ by
				$$
				\widehat{T_{m_{j,t}}f}(\xi)
				=m_{j,t}(\xi)\widehat f(\xi),
				$$
				where $
				m_{j,t}(\xi)
				:=|\xi|\widetilde{\varphi}(2^{-j}\xi)
				(\Lambda m_0)(t\xi).
				$
				Let $K_{j,t}:=\widecheck{m_{j,t}}$ be the convolution kernel of $T_{m_{j,t}}$. Notice that, by construction,
				$$
				\Lambda_t g_j(t,\cdot)
				=
				T_{m_{j,t}}(\Delta_jA').
				$$
				Then, by Young's inequality,
				$$
				\|\Lambda_t g_j(t,\cdot)\|_{L^1_x}
				\le
				\|K_{j,t}\|_{L^1}\|\Delta_jA'\|_{L^1}.
				$$
				Thus, it suffices to estimate $\|K_{j,t}\|_{L^1}$. Since by definition
				$$
				m_0(u)=\int_0^1 e^{2\pi i u \theta}\,d\theta,
				$$
				$\displaystyle\widehat{m_0}(\xi)=\chi_{[0,1]}(\xi)$. Therefore, the Fourier transform of \(\Lambda m_0\) is \(\xi \chi_{[0,1]}(\xi)\), and by the inversion formula we obtain
				$$
				\Lambda m_0(u)=\int_0^1 \theta e^{2\pi i u\theta}\,d\theta.
				$$
				
				Thus, the multiplier may be written as
				\[
				m_{j,t}(\xi)
				= |\xi|\widetilde{\varphi}(2^{-j}\xi)
				\int_0^1 \theta e^{2\pi i t\xi\theta}\,d\theta.
				\]
				Now set
				$\displaystyle
				\eta_j(\xi):=|\xi|\widetilde{\varphi}(2^{-j}\xi)$ and 
				$k_j:=\widecheck{\eta_j}.$
				Observe that by definition of $K_{j,t}$,
				\begin{align*}
					K_{j,t}(x)
					&=\int_{\R} e^{2\pi i x\xi}\eta_j(\xi)
					\int_0^1 \theta e^{2\pi i t\xi\theta}\,d\theta\,d\xi =\int_0^1 \theta
					\int_{\R} e^{2\pi i (x+t\theta)\xi}\eta_j(\xi)\,d\xi\,d\theta =\int_0^1 \theta k_j(x+t\theta)\,d\theta.
				\end{align*}
				
				To study $k_j$, we define $a(u):=|u|\widetilde{\varphi}(u)$, so that $ \eta_j(\xi)=2^j a(2^{-j}\xi)$. By the scaling rule for the inverse Fourier transform,
				\[
				k_j(x)=\widecheck{\eta_j}(x)=2^{2j}\widecheck{a}(2^j x).
				\]
				Therefore
				\[
				\|k_j\|_{L^1}
				=\int_{\R}2^{2j}|\widecheck{a}(2^j x)|\,dx
				=2^j\|\widecheck{a}\|_{L^1}.
				\]
				Since $a\in \mathcal{C}_c^\infty(\R)$, we have $\widecheck a\in\mathcal S(\R)$, and therefore $\|\widecheck a\|_{L^1}<\infty$. Consequently,
				\[
				\|k_j\|_{L^1}\lesssim 2^j.
				\]
				Using the representation of $K_{j,t}$,
				\begin{align*}
					\|K_{j,t}\|_{L^1}
					&\leq \int_0^1 \theta \|k_j(\cdot+t\theta)\|_{L^1}\,d\theta =\left(\int_0^1\theta\,d\theta\right)\|k_j\|_{L^1} \simeq 2^j.
				\end{align*}
				This proves the first bound, uniformly for all $t\in\R$, including $t=0$.
				
				To obtain the improved estimate for $|t|\neq 0$,  we write
				$k_j$ as a derivative of a kernel whose $L^1$ norm is independent of $j$.
				Define
				\[
				b_j(\xi)
				:=\frac{|\xi|}{2\pi i\xi}\widetilde{\varphi}(2^{-j}\xi)
				=\frac{\operatorname{sgn}\xi}{2\pi i}\widetilde{\varphi}(2^{-j}\xi).
				\]
				This is a smooth compactly supported function because
				$\widetilde{\varphi}(2^{-j}\xi)$ is supported away from $\xi=0$.
				Let $\ell_j:=\widecheck{b_j}$. Since $2\pi i\xi b_j(\xi)=|\xi|\widetilde{\varphi}(2^{-j}\xi)=\eta_j(\xi)$, we have
				\[
				k_j= \ell_j'.
				\]
				Now define
				\[
				b(\eta):=\frac{\operatorname{sgn}\eta}{2\pi i}\widetilde{\varphi}(\eta).
				\]
				Then $b_j(\xi)=b(2^{-j}\xi)$, so another application of the scaling rule gives $\ell_j(x)=2^j\widecheck b(2^j x)$. Therefore
				\[
				\|\ell_j\|_{L^1}
				=\int_{\R}2^j|\widecheck b(2^j x)|\,dx
				=\|\widecheck b\|_{L^1}.
				\]
				Since $b\in \mathcal{C}_c^\infty(\R)$, we have $\widecheck b\in\mathcal S(\R)$, hence $ \|\ell_j\|_{L^1}\lesssim 1$ uniformly in $j$. Using \(k_j=\ell_j'\), we can write
				\[
				K_{j,t}(x)
				=
				\int_0^1 \theta \ell_j'(x+t\theta)\,d\theta.
				\]
				
				Let us assume that \(t>0\) (the arguments for $t<0$ are analogous). For fixed \(x\), make the change of variables $u=x+t\theta$. Then, integration by parts yields
				\begin{align*}
					K_{j,t}(x)
					&=
					\frac1{t^2}\int_x^{x+t} (u-x)\ell_j'(u)\,du\\
					&= \frac1t\ell_j(x+t)
					-
					\frac1{t^2}\int_x^{x+t}\ell_j(u)\,du.
				\end{align*}
				We now estimate the \(L^1\)-norm. The first term satisfies
				\[
				\left\|\frac1t\ell_j(\cdot+t)\right\|_{L^1}
				=
				\frac1t\|\ell_j\|_{L^1}.
				\]
				For the second term, by Fubini's theorem,
				\[
				\begin{aligned}
					\int_{\R}
					\left|
					\frac1{t^2}\int_x^{x+t}\ell_j(u)\,du
					\right|dx
					&\leq
					\frac1{t^2}\int_{\R}\int_x^{x+t}|\ell_j(u)|\,du\,dx \\
					&=
					\frac1{t^2}\int_{\R}|\ell_j(u)|
					\bigl|\{x\in\R: x\leq u\leq x+t\}\bigr|\,du \\
					&=\frac{1}{t}\int_{\R}|\ell_j(u)|\,du =\frac{1}{t}\|\ell_j\|_{L^1}.
				\end{aligned}
				\]
				
				Combining the two estimates,
				\[
				\|K_{j,t}\|_{L^1}
				\leq
				\frac2t\|\ell_j\|_{L^1}
				\lesssim t^{-1}.
				\]
				
				Therefore, we have that $\|K_{j,t}\|_{L^1}
				\lesssim 2^j$ for all $t$, and that $\|K_{j,t}\|_{L^1}
				\lesssim |t|^{-1}$ for $t\neq 0.$ Hence
				$$
				\|\Lambda_t g_j(t,\cdot)\|_{L^1_x}
				\lesssim
				\min\{2^j,|t|^{-1}\}\|\Delta_jA'\|_{L^1}.
				$$
				This proves the desired estimate for $j\geq0$.
				
				For $j=-1$, the same representation holds with $\eta_{-1}(\xi)=|\xi|\varphi_{-1}(\xi)$. Its inverse Fourier transform belongs to $L^1(\mathbb R)$, since $\eta_{-1}$ is compactly supported and has only a Lipschitz singularity at the origin. Hence the associated kernel has uniformly bounded $L^1$ norm, and therefore $	\|\Lambda_t g_{-1}(t,\cdot)\|_{L^1_x}
				\lesssim
				\|\Delta_{-1}A'\|_{L^1}.$
			\end{proof}
			We now return to  \eqref{eq4.4}. For $j\geq 0$, multiplying by $\psi^2(t)$ and integrating in $t$, we get by Lemma \ref{lem7.5}
			\begin{align*}
				\|\psi^2(t)\Lambda_t g_j\|_{L^1}
				&\lesssim
				\|\Delta_j A'\|_{L^1}
				\int_{|t|\le3}
				\bigl(1+\min\{2^j,|t|^{-1}\}\bigr)\,dt \lesssim
				(1+j)\|\Delta_j A'\|_{L^1}.
			\end{align*}
			For $j=-1$ we simply have $\|\psi^2(t)\Lambda_t g_j\|_{L^1}\lesssim \|\Delta_{-1}A'\|_{L^1}\lesssim 1$, since $A'$ is supported in $[0,1]$ with $\|A'\|_\infty = 1$. Therefore, recalling definition \eqref{norm_Besov},
			$$
			\|\psi^2(t)\Lambda_t G\|_{L^1}
			\le
			\sum_{j\ge -1}
			\|\psi^2(t)\Lambda_t g_j\|_{L^1}
			\lesssim 1+
			\sum_{j\ge 0}
			(1+j)\|\Delta_j A'\|_{L^1}=
			\|A'\|_{B^{0,1}_{1,1}}.
			$$
			Returning to \eqref{eq4.4}, we conclude that
			$$
			\|\partial_t^{1/2}F_n\|_{L^2}^2
			\lesssim
			n\left(
			1+
			\|A'\|_{B^{0,1}_{1,1}}
			\right).
			$$
			Taking square roots and applying Lemma~\ref{3}, we obtain the estimate of Theorem~\ref{5}.
			
			\begin{rem}
				Let us prove that functions with support in $[0,1]$ that belong to the Sobolev spaces $H^s(\R), 0<s<1,$ also belong to $B^{0,1}_{1,1}(\R)$. To do so, we will use that $H^s(\R)$ can be identified with the Besov space $B^s_{2,2}(\R)$, that is endowed with the norm
				\[
				\left( \,\sum_{j\geq 0} 2^{2js}\|\Delta_j f\|_{L^2}^2 \right)^{1/2}.
				\]
				Indeed, simply notice that for a function $f$, the expression $\Delta_j f$ is given by a convolution against a Schwartz kernel $\Delta_j f = K_j\ast f$, with $K_j(x)=2^jK(2^jx)$. Then, by splitting the norm $\|\Delta_j f\|_{L^1}$ in $\|\Delta_j f\|_{L^1(-1,2)}$ and $\|\Delta_j f\|_{L^1(-1,2)^c}$, using that $\mathrm{supp}(f)\subset [0,1]$ and the rapid decay of $K$, it is not hard to prove that
				\[
				\|\Delta_j f\|_{L^1} \lesssim \|\Delta_j f\|_{L^2}+2^{-j}\|f\|_{L^2}.
				\]
				Then,
				\[
				\sum_{j\geq 0} (1+j)\|\Delta_j f\|_{L^1} \lesssim \sum_{j\geq 0} (1+j)\|\Delta_j f\|_{L^2} + \|f\|_{L^2}\sum_{j\geq 0}(1+j)2^{-j}.
				\]
				The second sum is finite. For the first, we simply use Cauchy-Schwarz and obtain
				\[
				\sum_{j\geq 0} (1+j)\|\Delta_j f\|_{L^2}\leq \left( \sum_{j\geq 0} (1+j)^22^{-2js} \right)^{1/2}\left( \sum_{j\geq 0} 2^{2js}\|\Delta_j f\|_{L^2} \right)^{1/2}.
				\]
				Since the first factor is finite for $s>0$, we deduce
				\[
				\sum_{j\geq 0} (1+j)\|\Delta_j f\|_{L^1}\lesssim_s \|f\|_{B^s_{2,2}}.
				\]
				Thus, for compactly supported functions, the condition
				\(A'\in B^{0,1}_{1,1}(\R)\) applies to a strictly larger class than the Sobolev condition appearing in Corollary~\ref{4}.
				Moreover, for compactly supported functions, the class is genuinely larger, since \(BV\) functions belong to \(B^{0,1}_{1,1}(\R)\). To see this, we introduce the Besov space $\textbf{B}^{0,1}_{1,1}(\R)$, defined via the norm
				\begin{equation*}
					\|f\|_{L^1}+\int_0^1 \frac{\omega_1(f,t)}{t}\log{\frac{e}{t}}dt,
				\end{equation*}
				where $\omega_1(f,t):=\sup_{|h|\leq t}\|f(\cdot+h)-f(\cdot)\|_{L^1}$. By \cite[Theorem 3.9]{CD2}, one has the embedding $\textbf{B}^{0,1}_{1,1}\hookrightarrow B^{0,1}_{1,1}$. Moreover, by \cite[Theorem 13.48]{L}, we also have
				\begin{equation*}
					\omega_1(f,t)\leq C\, t\|f\|_{BV}.
				\end{equation*}
				The previous bound is clear if one takes $f\in \mathcal{C}^1$ and identifies $\|f\|_{BV}$ with $\|f'\|_{L^1}$. The result for $BV$ found in \cite{L} follows, essentially, by an approximation argument. In any case, we get that for a function of bounded variation supported in $[0,1]$,
				\[
				\|f\|_{B^{0,1}_{1,1}}\lesssim \|f\|_{L^1}+\int_0^1 \frac{\omega_1(f,t)}{t}\log{\frac{e}{t}}dt \lesssim \|f\|_{L^1}+\|f\|_{BV},
				\]
				that is a finite quantity. Therefore, for compactly supported functions $BV\subset B^{0,1}_{1,1}$.
				
			\end{rem}
			\section{Incomparability of the regularity assumptions}
			
			\label{examples}
			In this section we show that the assumptions in Theorem~\ref{2} and Corollary~\ref{4} are incomparable by constructing two Lipschitz functions supported in $[0,1]$: one belongs to $\mathcal{C}^{1,\mathrm{Dini}}$ but does not satisfy $A'\in H^{1/2}(\mathbb R)$, while the other satisfies $A'\in H^{1/2}(\mathbb R)$ but does not belong to $\mathcal{C}^{1,\mathrm{Dini}}$. We choose the fractional parameter of the Sobolev space to be $1/2$ for simplicity. The arguments that follow can be adapted to obtain examples for $H^s(\R), 0<s\leq 1/2$.
			
			We begin with the first example. Define for $x\in[0,1]$ the function
			\begin{equation*}
				g_1(x):=\sum_{k=2}^\infty \frac{1}{k^3}\sin{(2\pi2^kx)},
			\end{equation*}
			as well as the function
			\begin{equation*}
				A_1(x):=\bigg(\int_0^x g_1(t)d t\bigg)\chi_{[0,1]}(x).
			\end{equation*}
			Since the series defining $g_1$ converges uniformly, we obtain that $g_1$ is continuous and it also satisfies $g_1(0)=g_1(1)=0$. Moreover, $A_1(0)=A_1(1)=0$ and $A_1'=g_1$ in $(0,1)$. Then, $A_1$ is in fact a $\mathcal{C}^1$ function supported in $[0,1]$. Let $0<h<\frac{1}{2}$ and take a positive integer $m$ with $2^{-m-1}<h<2^{-m}$. Then, for every $x$,
			\begin{equation*}
				|g_1(x+h)-g_1(x)|\leq \sum_{k=2}^\infty \frac{|\sin{(2\pi2^k(x+h))}-\sin{(2\pi2^kx)}|}{k^3}.
			\end{equation*}
			We distinguish whether $k\leq m$ or $k>m$. For the first case, by the mean value theorem,
			\begin{align*}
				\sum_{k\leq m} \frac{|\sin{(2\pi2^k(x+h))}-\sin{(2\pi2^kx)}|}{k^3} &\leq Ch\sum_{k\leq m} \frac{2^k}{k^3} = Ch \frac{2^m}{m^3} \sum_{k\leq m} 2^{k-m}\bigg(\frac{m}{k}\bigg)^3\\
				&\leq  Ch \frac{2^m}{m^3} < \frac{C}{m^3}.
			\end{align*}
			For $k>m$, we simply have
			\begin{equation*}
				\sum_{k > m} \frac{|\sin{(2\pi2^k(x+h))}-\sin{(2\pi2^kx)}|}{k^3} \leq \sum_{k>m}\frac{2}{k^3} \leq 2\int_m^\infty \frac{d x}{x^3} = \frac{1}{m^2}.
			\end{equation*}
			Therefore, if $\omega_{g_1}$ denotes the modulus of continuity of $g_1$, we get
			\begin{equation*}
				\omega_{g_1}(h)\lesssim \frac{1}{m^2} \simeq \frac{1}{\log(1/h)^2}, \quad \text{so that} \quad \int_0^{1/2} \frac{\omega_{g_1}(h)}{h}d h \lesssim \int_0^{1/2}\frac{d h}{h\log{(1/h)}^2}<\infty.
			\end{equation*}
			Hence $A_1\in \mathcal{C}^{1,\mathrm{Dini}}$. On the other hand, we show that the function $A'_1\notin H^{1/2}(\R)$. To do so, it suffices to prove that the following integral diverges:
			\begin{equation*}
				\int_{\R}|\xi||\widehat{A'_1}(\xi)|^2d \xi.
			\end{equation*}
			Define the function
			\begin{equation*}
				S(t):=e^{-\pi i t}\frac{\sin{(\pi t)}}{\pi t},
			\end{equation*}
			with $S(0):=1$. Using that the Fourier transform of $\chi_{[-\frac{1}{2},\frac{1}{2}]}(x)$ is given by $\frac{\sin(\pi \xi)}{\pi \xi}$, it is not hard to see that the Fourier transform of $\chi_{[0,1]}(x)e^{2\pi i mx}$ is given by $S(\xi -m)$. Therefore,
			\begin{align*}
				\widehat{A'_1}(\xi) =\left( \sum_{k=2}^\infty \frac{1}{k^3}\sin{(2\pi 2^k x)} \right)^{\wedge}(\xi) = \sum_{k=2}^\infty \frac{1}{2ik^3}\left(S(\xi-2^k)-S(\xi+2^k)\right).
			\end{align*}
			For every $m\geq 2$ we define the interval $I_m:=[2^m-\frac{1}{8}, 2^m+\frac{1}{8}]$. Observe that if $m\neq m'$, then $I_m\cap I_{m'}= \varnothing$. If $\xi\in I_m$, then $|\xi-2^m|\leq \frac{1}{8}$, so $|S(\xi-2^m)|\geq c$ for a certain constant $c$ (since the $\text{sinc}$ function is continuous and equals 1 at the origin) and $|S(\xi+2^m)|\leq \frac{C}{2^m}$, for another positive constant $C$. Therefore, for $m$ big enough and every $\xi\in I_m$, we get that the $m$-th term of the sum defining $\widehat{A'_1}$ satisfies
			\begin{equation*}
				\bigg\rvert \frac{1}{2i m^3}\left(S(\xi-2^m)-S(\xi+2^m)\right)\bigg\rvert \geq \frac{c_0}{m^3}.
			\end{equation*}
			for some $c_0>0$. Let now $k\neq m$ and $\xi\in I_m$. If $k<m$, then $|\xi-2^k|\gtrsim 2^m$, and if $k>m$, then $|\xi-2^k|\gtrsim 2^k$. Moreover, we always have $|\xi+2^k|\gtrsim 2^{\max\{k,m\}}$. Since $|S(t)|\leq \frac{2}{1+|t|}$, we deduce
			$$
			|S(\xi-2^k)|+|S(\xi+2^k)|\leq \begin{cases*}
				C2^{-m} & \text{if } $k < m$, \\
				C2^{-k}  & \text{if } $k > m$.
			\end{cases*}
			$$
			for some constant $C>0$. Therefore,
			\begin{equation*}
				\sum_{k\neq m} \frac{1}{k^3}\left( |S(\xi-2^k)|+|S(\xi+2^k)| \right) \leq C\bigg( 2^{-m} \sum_{k<m} \frac{1}{k^3} + \sum_{k > m} \frac{2^{-k}}{k^3} \bigg) \leq C_0 2^{-m},
			\end{equation*}
			for some $C_0>0$. Now, we choose $m_0$ big enough so that $C_02^{-m}\leq \frac{c_0}{2m^3}$ for every $m\geq m_0$. Hence, we conclude that for every $\xi\in I_m$ with $m\geq m_0$, there exists a constant $c_1>0$ so that 
			\begin{equation*}
				|\widehat{A'_1}(\xi)|\geq \frac{c_1}{m^3}.
			\end{equation*}
			Therefore,
			\begin{align*}
				\int_{\R}|\xi||\widehat{A'_1}(\xi)|^2d \xi \geq \sum_{m\geq m_0} \int_{I_m} |\xi||\widehat{A'_1}(\xi)|^2d \xi \geq \sum_{m\geq m_0} \left(2^m-\frac{1}{8}\right)|I_m|\frac{c_1^2}{m^6}=\infty,
			\end{align*}
			and we deduce $A_1'\not\in H^{1/2}(\R)$.
			
			\medskip
			Let us now construct a second function $A_2$ that satisfies $A_2'\in H^{1/2}(\mathbb R)\subset B^{0,1}_{1,1}(\R)$ but $A_2\notin \mathcal{C}^{1,\mathrm{Dini}}$. Define for every $x\in[0,1]$,
			\begin{equation*}
				g_2(x):=\sum_{k=2}^\infty \frac{\sin{(2\pi kx)}}{k\log{(k+1)}},
			\end{equation*}
			as well as
			\begin{equation*}
				A_2(x):=\bigg(\int_0^x g_2(t)d t\bigg)\chi_{[0,1]}(x).
			\end{equation*}
			As in the previous example, $g_2$ is continuous and $A_2$ is a $\mathcal{C}^1$ function supported in $[0,1]$, with $A_2'=g_2$ in $(0,1)$. 
			
			We first show that $A_2'\in H^{1/2}(\R)$. For this, we use the spectral definition of fractional Sobolev spaces via powers of the Dirichlet Laplacian. Namely, if
			\begin{equation*}
				u(x)=\bigg(\sum_{m\ge1} c_m \sin(m\pi x)\bigg)\chi_{[0,1]}(x),
			\end{equation*}
			then 
			\begin{equation}
				\label{eq8.1}
				u\in H^{1/2}(\mathbb{R}) \qquad \text{if and only if} \qquad \sum_{m=1}^\infty m|c_m|^2 <\infty.
			\end{equation}
			We will justify this equivalence after the construction of the example. For the moment, let us accept this characterization and set $u:=A'_2$. Since
			\begin{equation*}
				\sum_{k=2}^{\infty} 2k\bigg(\frac{1}{k\log{(k+1)}}\bigg)^2<\infty,
			\end{equation*}
			$A_2'\in H^{1/2}(\mathbb{R})$. We now show that $A_2\not\in \mathcal{C}^{1,\mathrm{Dini}}$. Indeed, take $h_N:=\frac{1}{8N}$, for some $N$ positive integer. Then, for $1\leq k \leq N$,
			\begin{equation*}
				0\leq 2\pi k h_N = \frac{\pi k}{4N}\leq \frac{\pi}{4}.
			\end{equation*}
			Since $\sin{u}\geq \frac{u}{2}$ for $0\leq u \leq \frac{\pi}{4}$, we get
			\begin{align*}
				g_2(h_N)=\sum_{k=2}^\infty \frac{\sin{(2\pi kh_N)}}{k\log{(k+1)}}\geq \sum_{k=2}^N \frac{\sin{(2\pi kh_N)}}{k\log{(k+1)}} \geq \pi h_N \sum_{k=2}^N \frac{1}{\log{(k+1)}}\geq Ch_N\frac{N}{\log{N}}\simeq \frac{1}{\log{N}}.
			\end{align*}
			Then, since $g_2(0)=0$, if $\omega_{g_2}$ denotes the modulus of continuity of $g_2$,
			\begin{equation*}
				\omega_{g_2}(h_N)\geq |g_2(h_N)-g_2(0)|\gtrsim \frac{1}{\log{N}}\simeq \frac{1}{\log{\frac{1}{h_N}}}.
			\end{equation*}
			Now choose $N:=2^j, j\geq 0,$  and set $h_j:=h_{2^{j}}=2^{-j-3}$. Then, $\omega_{g_2}(h_j)\gtrsim \frac{1}{j}$, and since $\omega_{g_2}$ is non-decreasing,
			\begin{equation*}
				\int_{h_{j+1}}^{h_j} \frac{\omega_{g_2}(h)}{h}d h \geq \omega_{g_2}(h_{j+1}) \log{\frac{h_j}{h_{j+1}}} \gtrsim \frac{1}{j}.
			\end{equation*}
			Therefore
			\begin{equation*}
				\int_0^{h_1} \frac{\omega_{g_2}(h)}{h} d h \gtrsim \sum_{j=1}^\infty \frac{1}{j} = \infty,
			\end{equation*}
			yielding that $A_2\not\in \mathcal{C}^{1,\mathrm{Dini}}$.
			
			\begin{rem}
				This last example also shows that $B^{0,1}_{1,1}(\R)\not\subset \mathcal{C}^{\mathrm{Dini}}(\R)$ for compactly supported functions. However, we have not been able to construct an example of a compactly supported function belonging to $\mathcal{C}^{\mathrm{Dini}}(\R)$ and not to $B^{0,1}_{1,1}(\R)$. 
			\end{rem}
			
			\subsection{Validation of identity \eqref{eq8.1}}
			
			Finally, we prove the following result:
			\begin{thm}
				\label{thm8.1}
				Let $v(x)=\sum_{m=1}^{\infty} c_m \sin(\pi m x)$ in $L^2(0,1)$, and let $u : \mathbb{R} \to \mathbb{C}$ be its zero-extension onto $\mathbb{R}$. Then $u \in H^{1/2}(\mathbb{R})$ if and only if
				\begin{equation*}
					\sum_{m=1}^{\infty} m|c_m|^2 < \infty.
				\end{equation*}
			\end{thm}
			To establish the result, let $I$ be an interval and begin by recalling the equivalent local characterization of the Sobolev norm $H^{1/2}(I)$, given by the Gagliardo semi-norm:
			\begin{equation*}
				\|v\|^2_{H^{1/2}(I)} = \|v\|^2_{L^2(I)} + \int_{I}\int_{I} \frac{|v(x)-v(y)|^2}{|x-y|^{2}}dx\,dy.
			\end{equation*}
			
			%standard definitions and properties of fractional Sobolev spaces. As we did for the case $s=1/2$ in \eqref{def6.2}, we define more generally for any $s \in \mathbb{R}$ the \textit{Sobolev space} $H^s(\mathbb{R})$ as the space of tempered distributions $u \in \mathcal{S}'(\mathbb{R})$ whose Fourier transform satisfies:
			%\begin{equation*}
			%\|u\|^2_{H^s(\mathbb{R})} := \int_{\mathbb{R}} (1+|\xi|^2)^s |\widehat{u}(\xi)|^2 d\xi < \infty.
			%\end{equation*}
			
			%For $s \in (0,1)$, an equivalent and local characterisation is given by the Gagliardo semi-norm. Namely, $u \in H^s(\mathbb{R})$ if and only if $u \in L^2(\mathbb{R})$ and the Gagliardo seminorm is finite:
			%\begin{equation*}
			%[u]^2_{H^s(\mathbb{R})} := \int_{\mathbb{R}}\int_{\mathbb{R}} \frac{|u(x)-u(y)|^2}{|x-y|^{1+2s}}dx\,dy < \infty.
			%\end{equation*}
			
			%The total norm $ \|u\|^2_{L^2(\mathbb{R})} + [u]^2_{H^s(\mathbb{R})}$ is equivalent to the Fourier-defined norm $\|u\|^2_{H^s(\mathbb{R})}$. The proof of this equivalence relies on the change of variables $t:=x-y$ and on applying essentially the same Fourier transform methods used to prove identity \eqref{identity}.
			
			%This formulation of the norm allows one to define Sobolev spaces on subdomains $\Omega\subset \R$, provided they are sufficiently regular. For a bounded open interval $I \subset \mathbb{R}$, one may define 
			%\begin{equation*}
			%\|v\|^2_{H^s(I)} = \|v\|^2_{L^2(I)} + \int_{I}\int_{I} \frac{|v(x)-v(y)|^2}{|x-y|^{1+2s}}dx\,dy,
			%\end{equation*}
			%and convey that  $v\in H^s(I)$ if there exists $u\in H^s(\mathbb{R})$ such that $\|v-u\|_{H^{s}(I)}=0$.

			Let us now study how the closure of smooth functions behaves in the $H^{1/2}(I)$ norm. Let us convey that $I=(0,1)$ for simplicity.
			
			\begin{defn*}
				We define the space $H^{1/2}_{00}(0,1)$ as:
				\begin{equation*}
					H^{1/2}_{00}(0,1) := \left\{ v \in H^{1/2}(0,1) : \int_0^1 \frac{|v(x)|^2}{x(1-x)}dx < \infty \right\}.
				\end{equation*}
			\end{defn*}
			
			The fundamental link between extension operators and $H^{1/2}_{00}$ is provided by the following classical theorem due to Lions and Magenes \cite[Theorem 11.7]{LM}.
			
			\begin{thm}
				\label{thm2}
				Let $v \in L^2(0,1)$, and let $\widetilde{v}$ denote its extension by zero to $\mathbb{R}$. Then $\widetilde{v} \in H^{1/2}(\mathbb{R})$ if and only if $v \in H^{1/2}_{00}(0,1)$. Moreover,
				\begin{equation*}
					\|\widetilde{v}\|_{H^{1/2}(\mathbb{R})} \simeq \|v\|_{H^{1/2}_{00}(0,1)}.
				\end{equation*}
			\end{thm}
			
			We introduce basic terminology about fractional powers of the Dirichlet Laplacian. In our $L^2(0,1)$ setting, the Dirichlet Laplacian is simply $A=-\Delta_D=-\frac{d^2}{dx^2}$ and acts on the natural domain:
			\begin{equation*}
				\text{dom}(A)=H^2(0,1)\cap H^1_0(0,1),
			\end{equation*}
			
			where $H^1_0(0,1)$ denotes the closure of $\mathcal{C}^\infty_0(\R)$ functions that belong to $H^1(0,1)$. Any function $v \in L^2(0,1)$ can be written uniquely as $\displaystyle v(x)=\sum_{m=1}^{\infty} c_m \sin(\pi m x)$. Using the Spectral Theorem, we can define the fractional power $A^{\alpha}$ for any $\alpha >0$ as 
			\begin{equation*}
				A^{\alpha}v = \sum_{m=1}^{\infty} \lambda_m^{\alpha}c_m \sin{(\pi m x)},
			\end{equation*}
			where $\lambda_m = \pi^2 m^2$ are the eigenvalues of the normalized eigenfunctions of $A$, $e_m(x)=\sqrt{2}\sin(\pi m x)$ for $m\in\mathbb{N}$. The domain of $A^\alpha$ is
			\begin{equation*}
				\text{dom}(A^{\alpha})=\left\{ v \in L^2(0,1) : \sum_{m=1}^{\infty} \lambda_m^{2\alpha}|c_m|^2 < \infty \right\}.
			\end{equation*}
			
			For $\alpha=1/2$ it is not hard to prove that $\text{dom}(A^{1/2})=H^1_0(0,1)$, since for a function $v(x)=\sum_{m=1}^{\infty} c_m \sin(\pi m x)$, the sum $\sum_{m=1}^{\infty} m^2|c_m|^2$ should be finite. Moreover, for $\alpha = 1/4$ the condition is $\sum_{m=1}^{\infty} m|c_m|^2 <\infty $.
			
			The final step requires identifying the spectral domain $\text{dom}(A^{1/4})$ with the Lions-Magenes space $H^{1/2}_{00}(0,1)$. This is done by standard interpolation methods of domains of self-adjoint operators. In particular, we have, by definition
			\begin{equation*}
				\text{dom}(A^{1/4})=[\text{dom}(A^{1/2}),\text{dom}(A^{0})]_{1/2}=[H^1_0(0,1),L^2(0,1)]_{1/2}
			\end{equation*}
			The characterisation of the interpolated space $[H^1_0(0,1),L^2(0,1)]_{1/2}$ is also given in \cite[Theorem 11.7]{LM}, and is precisely
			\begin{equation*}
				[H^1_0(0,1),L^2(0,1)]_{1/2}=H^{1/2}_{00}(0,1),
			\end{equation*}
			with equivalent norms. Therefore, we deduce the identity:
			\begin{equation}
				\label{eq0.1}
				v \in H^{1/2}_{00}(0,1) \Longleftrightarrow v \in \text{dom}(A^{1/4}) \Longleftrightarrow \sum_{m=1}^{\infty} m|c_m|^2 < \infty.
			\end{equation}
			
			We are now ready to combine the previous results to prove Theorem \ref{thm8.1}.
			
			\begin{proof}
				Let $v(x)=\sum_{m=1}^{\infty} c_m \sin(\pi m x)$ be an element of $L^2(0,1)$ and let $u(x)=v(x)\chi_{[0,1]}(x)$ be its extension by zero to the real line $\mathbb{R}$.
				
				\noindent$(\Rightarrow)$ Suppose that $u \in H^{1/2}(\mathbb{R})$. By definition, $u$ is the zero-extension of $v$ from the interval $(0,1)$ to $\mathbb{R}$. Applying Theorem \ref{thm2}, we deduce that $v$ must belong to $H^{1/2}_{00}(0,1)$, so by the equivalences in \eqref{eq0.1} we are done.
				
				\noindent$(\Leftarrow)$ Conversely, assume that the coefficient sequence satisfies $\sum_{m=1}^{\infty} m|c_m|^2 < \infty$. By \eqref{eq0.1}, $v \in H^{1/2}_{00}(0,1)$. Then, by Theorem \ref{thm2}, since $v \in H^{1/2}_{00}(0,1)$, its zero-extension $\widetilde{v}=u$ automatically belongs to $H^{1/2}(\mathbb{R})$.
			\end{proof}

			\vspace{1.5cm}
			{\small
				\begin{tabular}{@{}l}
					\textsc{Joan\ Hernández,} \\ \textsc{Departament de Matem\`{a}tiques, Universitat Aut\`{o}noma de Barcelona,}\\
					\textsc{08193, Bellaterra (Barcelona), Catalonia.}\\
					{\it E-mail address}\,: \href{mailto:joan.hernandez@uab.cat}{\tt{joan.hernandez@uab.cat}}
				\end{tabular}
			}
			
			{\small
				\begin{tabular}{@{}l}
					\textsc{Joan\ Mateu,} \\ \textsc{Departament de Matem\`{a}tiques, Universitat Aut\`{o}noma de Barcelona,}\\
					\textsc{08193, Bellaterra (Barcelona), Catalonia.}\\
					{\it E-mail address}\,: \href{mailto:joan.mateu@uab.cat}{\tt{joan.mateu@uab.cat}}
				\end{tabular}
			}

			{\small
				\begin{tabular}{@{}l}
					\textsc{Laura\ Prat,} \\ \textsc{Departament de Matem\`{a}tiques, Universitat Aut\`{o}noma de Barcelona,}\\
					\textsc{08193, Bellaterra (Barcelona), Catalonia.}\\
					{\it E-mail address}\,: \href{mailto:laura.prat@uab.cat}{\tt{laura.prat@uab.cat}}
				\end{tabular}
			}  
			
		\end{document}